\documentclass[]{interact}

\usepackage{epstopdf}
\usepackage{subfigure}

\usepackage{caption}
\usepackage{hyperref}
\usepackage{subcaption}
\usepackage{array}
\usepackage{amsmath,amsthm,amssymb,comment,fullpage}
\usepackage{graphicx}
\usepackage{enumitem}
\usepackage{bm}
\usepackage{tikz}
\usepackage{tikz-qtree,tikz-qtree-compat}
\usepackage{tkz-tab}
\usepackage{tkz-graph}
\usetikzlibrary{shapes.geometric,positioning}

\theoremstyle{plain}
\newtheorem{thm}{Theorem}[section]
\newtheorem{corollary}[thm]{Corollary}
\newtheorem{lemma}[thm]{Lemma}
\newtheorem{rek}[thm]{Remark}

\newcommand{\ST}[1]{\textcolor{blue}{\bm{#1}}}
\newcommand{\NS}[1]{\textcolor{red}{{#1}}}
\newcommand{\pvs}{\par \vspace{.25cm}}

\newcommand{\tsub}[1]{\textsubscript{{#1}}}
\renewcommand{\a}{\alpha}
\renewcommand{\b}{\beta}
\newcommand{\g}{\gamma}

\newcommand{\Znn}{\ensuremath{\mathbb{Z}}^{\geq0}}
\newcommand{\Zp}{\ensuremath{\mathbb{Z}}^{
>0}}

\begin{document}

\title{Black Hole Zeckendorf Games}

\author{
\name{Caroline Cashman\textsuperscript{a}\thanks{CONTACT Caroline Cashman. Email: cecashman@wm.edu},  Steven J. Miller\textsuperscript{b}\thanks{CONTACT Steven J. Miller. Email: sjm1@williams.edu}, Jenna Shuffelton\textsuperscript{c}\thanks{CONTACT Jenna Shuffelton. Email: jms13@williams.edu}, and Daeyoung Son\textsuperscript{a}\thanks{CONTACT Daeyoung Son. Email: ds15@williams.edu},}
\affil{\textsuperscript{a}College of William \& Mary, Willliamsburg, VA; \textsuperscript{b}Williams College, Williamstown, MA; \textsuperscript{c}Williams College, Williamstown, MA; \textsuperscript{d}Williams College, Williamstown, MA}
}

\maketitle

\begin{abstract}
Zeckendorf proved that every positive integer can be written as a decomposition of non-adjacent Fibonacci numbers. Baird-Smith, Epstein, Flint, and Miller converted the process of decomposing an integer $n$ into a 2-player game, using the moves of $F_i + F_{i-1} = F_{i+1}$ and $2F_i = F_{i+1} + F_{i-2}$, where $F_i$ is the $i$\textsuperscript{th} Fibonacci number. They showed non-constructively that for $n \neq 2$, Player 2 has a winning strategy: a constructive solution remains unknown.

We expand on this by investigating ``black hole'' variants of this game. The $F_m$ Black Hole Zeckendorf game is played with any $n$ but solely in columns $F_i$ for $i < m$. Gameplay is similar to the original Zeckendorf game, except any piece that would be placed on $F_i$ for $i \geq m$ is locked out in a ``black hole'' and removed from play. With these constraints, we analyze the games with black holes on $F_3$ and $F_4$ and construct a solution for specific configurations, using a non-constructive proof to lead to a constructive one. We also examine a pre-game in which players take turns placing down $n$ pieces in the outermost columns before the decomposition phase, and find constructive solutions for any $n$.
\end{abstract}


\section{Introduction}
\subsection{Background} The beauty of the Fibonacci numbers is undeniable: a simple sequence, recursively defined by the sum of the two previous numbers, that has the tendency to show up in both natural and surprising places. Indexing so that $F_1=1$, $F_2=2$ and $F_{k+1}=F_k+F_{k-1}$,  Zeckendorf proved a particularly interesting fact about the Fibonacci numbers, namely that any positive integer $n$ can be written as the sum of non-adjacent Fibonacci numbers, known as the number's Zeckendorf decomposition \cite{Ze}. Baird-Smith, Epstein, Flint, and Miller \cite{BEFM1,BEFM2}, created a game from the process of converting a positive integer into its Zeckendorf decomposition using the moves of $F_i + F_{i-1} = F_{i+1}$ and $2F_i = F_{i+1} + F_{i-2}$, where $F_i$ is the $i$\textsuperscript{th} Fibonacci number. We outline the rules to the original Zeckendorf game as follows. \pvs
\begin{enumerate}
    \item \textbf{Setup:} The game is played on a board with columns corresponding to each of the Fibonacci numbers, indexing so that the $1$\textsuperscript{st} column corresponds with $F_1=1$, the $2$\textsuperscript{nd} column corresponds with $F_2=2$ and the $m$\textsuperscript{th} column corresponds with $F_m$, the $m$\textsuperscript{th} Fibonacci number. All $n$ pieces begin in the $1$\textsuperscript{st} column.
    \vspace{.25cm}
    \item \textbf{Gameplay:} Players alternate, selecting their moves from the following. \begin{enumerate}
        \item Adding consecutive terms: If the board contains pieces in both $F_i$ and $F_{i-1}$ columns, players can remove one piece from each column to add as one piece in the $F_{i+1}$ column.
        \item Merging 1's: If the board contains more than one piece in the $F_1$ column, players can remove two pieces from the $F_1$ column to merge as one piece in the $F_{2}$ column.
        \item Splitting: If the board contains more than one piece in the $F_2$ column, players can split two pieces from the $F_2$ column to place one piece in each $F_{1}$ and $F_3$. For $i\geq 3$, players can split  two pieces in the $F_{i}$ column to place one in each $F_{i-2}$ and $F_{i+1}$.
    \end{enumerate}
    \vspace{.25cm}
    \item \textbf{Winning:} The last player to move wins.
\end{enumerate} \pvs
They proved that the game is playable, meaning it always ends in finite time, and that the final board placed down will be equal to the Zeckendorf decomposition of $n$. Moreover, they showed that for all $n\neq 2$, Player $2$ has a winning strategy. Notably, this is not a constructive winning strategy, and instead relies on a parity stealing argument. If one assumes that Player $1$ has a winning strategy, Player $2$ later has the opportunity to steal it, therefore Player $2$ must have a winning strategy. With increasing $n$, the number of possible game positions grows exponentially, making the construction of a winning solution for Player $2$ challenging. Multiple variations of this game have been studied; see \cite{BCDD+15,BJMNSYY,CDHK+6i,CDHK+6ii,CMJDMMN,GMRVY,LLMMSXZ,MSY}. In order to develop a greater understanding of the original Zeckendorf game, we consider a variation occurring on a smaller board. \pvs
\subsection{Main Results} We consider an ``$F_m$ Black Hole'' variation of the Zeckendorf game, where once a piece is placed on some $F_i$ for $i\geq m$, it falls into the ``Zeckendorf Black Hole'' and is permanently removed from game play. This variant reduces the number of possible moves a player has, making the game easier to analyze.  We combine this Black Hole variation with an Empty Board variation, where the game begins with an empty board, and players take turns placing down pieces in the outermost columns until the weighted sum equals the starting value $n$. The last player to place down a piece moves second in the decomposition phase of the game, assuming the role of Player $2$ from the original Zeckendorf game. Combining these two variations is interesting for a variety of reasons. The solution to the $F_m$ Black Hole Zeckendorf game is heavily based on modular arithmetic, while the solution to the Empty Board Zeckendorf game uses a move mirroring strategy common in combinatorial game theory. \pvs
We quickly consider the Empty Board $F_m$ Black Hole Zeckendorf game with black holes on $F_1$, which is not playable, and on $F_2$, which is deterministic. We then shift our attention to the Empty Board $F_m$ Black Hole Zeckendorf game with black holes on $F_3$ and then on $F_4$. We determine which player has a winning strategy for any positive integer $n$, and provide a constructive solution.\pvs
\subsubsection{Terminology}
We first clarify some terminology. We refer to any column corresponding with the $i$\textsuperscript{th} Fibonacci number as the $F_i$ column. The number of pieces in a column at any given game state is $a$ for the $F_1$ column, $b$ for the $F_2$ column, and $c$ for the $F_3$ column, resulting in a game state $(a,b,c)$. Because our solutions are based on modular arithmetic, we also describe game states in terms of $\a,\b,\g$ and $k_1,k_2,k_3$, where $\a$ and $k_1$ correspond with the $F_1$ column, $\b$ and $k_2$ correspond with the $F_2$ column, and $\g$ and $k_3$ with the $F_3$ column. \pvs
To further clarify which player wins at a given position, we use the terms Player $1$ and Player $2$ only when talking about the game as a whole. For a given position within the game, we describe it as either $P$ or $N$, which is standard notation within combinatorial game theory; see \cite{ANW}, \cite{BCG}, \cite{BT}.  A certain game state is a $P$ position if the previous player who has placed it has a winning strategy. A game state is $N$ if the next player to play has a winning strategy. If a given position is $P$, all following possibilities must be $N$ positions, and if a given position is $N$ at least one of the following possibilities must be a $P$ position. When it is necessary to explain our game trees, we assume that Red is the first to move from a given position. 
\subsubsection{$F_3$ Results}
With a black hole on $F_3$, pieces can only be placed in the $F_1$ and $F_2$ columns. The winner of all possible games can be determined based on the value of $a$ and $b$ modulo $3$. We describe a board setup $(a,b)$ as $(3\a+k_1,3\b+k_2)$, where $\a,\b, k_1, k_2 \in \Znn$,  $0 \leq k_1, k_2 \leq 2$. We find that $(a,b)$ is a $P$ position for all $a\equiv b \equiv 0$, $a\equiv 0, b \equiv 1$ or $a\equiv 1, b \equiv 0$. Any other setup is an $N$ position. It follows that when placing $n$ pieces on an empty board, Player $1$ has a constructive winning strategy for $n\equiv 1,2,3,6,8\pmod 9$ and Player $2$ has a constructive winning strategy when $n \equiv 0,4,5,7 \pmod 9$.\pvs
\subsubsection{$F_4$ Results}
When the black hole moves to $F_4$, the game immediately becomes more interesting. The game is relatively straightforward for $(a,0,0)$, which is a $P$ position for all $a \neq 2$. The game is also relatively straightforward for $(0,0,c)$ which is an $N$ position for all $c \neq 0,1, 5$.\pvs
However, once the board is in the position $(a,0,c)$, the minor exceptions from the $(a,0,0)$ and the $(0,0,c)$ cases become incredibly influential in the game reduction. Still, we are able to simplify the game, and determine winners as outlined in Figure \ref{fig:winners(a,0,c)}. Then, we describe the board set up of $(a,0,c)$ as $(3\a+k_1,0,4\g+k_3)$ where $\a,\g, k_1, k_3 \in \Znn$, $0\leq k_1 \leq 2$, and $0 \leq k_3 \leq 3$. We show the winners based on the values of $\a$ and $\g$ below, denoting $P$ positions in bold blue and $N$ positions in red.\pvs
\begin{figure}[h!]
\begin{center}
\begin{tabular}{ | m{35mm}| m{35mm}| m{35mm}| m{35mm} | }
  \hline
  \text{ } & $a \equiv 0\pmod{3}$ & $a \equiv 1\pmod{3}$ & $a \equiv 2\pmod{3}$\\
  \hline
  $c \equiv 0\pmod{4}$ & $\ST{\a \geq \g}$ & $\ST{\forall \a,\g}$ & $\ST{\a \geq \g + 1}$\\
  & $\NS{\a \leq \g -1}$ &  & $\NS{\a \leq \g}$\\
  \hline
  $c \equiv 1\pmod4$& $\ST{\a \geq \g-1}$ & $\ST{\forall \a,\g}$ & $\ST{\a \geq \g}$\\
  & $\NS{\a \leq \g -2}$ &  & $\NS{\a \leq \g-1}$\\
   \hline
  $c \equiv 2\pmod4$& $\NS{\forall \a,\g}$ & $\NS{\a \geq \g +1}$& $\NS{\forall \a,\g}$  \\
  &  & $\ST{\a \leq \g }$&   \\
  \hline
  $c \equiv 3\pmod4$& $\NS{\forall \a,\g}$& $\NS{\a \geq \g}$ & $\NS{\forall \a,\g}$ \\
  & &  $\ST{\a \leq \g -1}$ & \\
  \hline
\end{tabular}
\end{center}
 \caption{Winners for board setups $(a,0,c)$ in an $F_4$ Black Hole Zeckendorf Game. $P$ positions are depicted in bold blue, and $N$ positions are depicted in red.}
    \label{fig:winners(a,0,c)}
\end{figure}
With this knowledge, we continue onto the Empty Board $F_4$ Black Hole Zeckendorf game, where players may only place in the outermost columns; we define the game in this way as the general case gives players far more options, which poses a variety of challenges for determining a constructive solution. We find that Player $1$ has a constructive winning solution for all $n \equiv 1,3,5,7,8,10,12,14,15 \pmod {16}$ such that $n\neq 17,47$, for which Player $2$ has a winning solution. Player $2$ has a constructive winning solution for all $n \equiv 0,2,4,6,9,11,13 \pmod {16}$ such that $n \neq 2,32$, for which Player $1$ has a winning solution. For large values of $n$, $\a \geq \g+1$, so it is possible to quickly determine winners by Figure \ref{fig:winners(a,0,c)}. For smaller values of $n$, it is necessary to explicitly consider certain values of $n$, as the winner is then dependent on the values of $k_1$ and $k_3$, which can lead to exceptions, as with $n=2,17,47,32$.\pvs
Moving the black hole to any $F_m$ such that $m \geq 5$ limits both players' ability to reduce down to a game with a known solution, meaning that a constructive solution is no longer immediately apparent. The fact that there exists no obvious constructive solution to a Zeckendorf game in as few as four columns supports the conjectured complexity of the original Zeckendorf game. \pvs

\section{Rules}
We define an Empty Board  $F_m$  Black Hole Zeckendorf game as follows.
\begin{enumerate}
    \item \textbf{Setup:} The game begins on an empty board with $m-1$ columns. The $F_i$ column is weighted to correspond with the $i$\textsuperscript{th} Fibonacci number, indexing so that $F_1=1$ and $F_2=2$. Players start the game with any positive integer $n$ pieces.
    \vspace{.25cm}
    \item \textbf{Placing the pieces:} Players take turns placing one piece in the outermost columns of the board, namely $F_1$ and $F_{m-1}$, until the weighted sum equals $n$. For us, this is equivalent to placing in columns $F_1$ for an initial board setup $(a)$, $F_1$ and $F_2$ for an initial board setup $(a,b)$ and in columns $F_1$ and $F_3$ for an initial board setup $(a,0,c)$\footnote{We are not able to provide a solution for the general case, as during the decomposition phase neither player is able to reduce the value of the board without giving the other player the option to do so first.}. Placing one piece in the $F_i$ column removes $F_i$ pieces from the pile of $n$, so players can only place in columns such that $F_i$ is less than or equal to the number of pieces left. This stage ends when there are no pieces left to be placed. The last player to place down a piece moves second in the decomposition phase of the game. Thus, it is the goal of both players to either set the board as a $P$ position, or to force their opponent to set the board as an $N$ position
    \vspace{.25cm}
    \item \textbf{Decomposition:} Players now begin the decomposition phase of the game. Players alternate, selecting their moves from the following.
    \begin{enumerate}
        \item \textit{Add:} Add one piece from each $F_i$ and $F_{i+1}$ to combine as one piece on $F_{i+2}$.
        \item \textit{Merge:} Merge two pieces from $F_1$ into one piece in column $F_2$.
        \item \textit{Split:} Split two pieces in column $F_2$ into one in each column $F_1$ and $F_3$ or for $i\geq 3$, split two pieces in column $F_{i}$ into one in each column $F_{i-2}$ and column $F_{i+1}$.
    \end{enumerate}
    \vspace{.25cm}
    \item \textbf{Black Hole:} Note that in the moves above, it is possible for pieces to be placed in $F_m$. In this situation, they become trapped in the ``Zeckendorf Black Hole'', where they are permanently removed from the board.
   \vspace{.25cm}
    \item \textbf{Winning:} The last player to move wins the game. In combinatorial game theory, this is commonly referred to as normal play.
\end{enumerate}\pvs
\begin{thm}
    The Empty Board Black Hole Zeckendorf game is playable and always ends at the Zeckendorf decomposition of $n \pmod{F_m}$.
\end{thm}
\begin{proof}
    The Empty Board portion of the game does not affect whether the game is playable, as there are a finite number of pieces to place, by definition always resulting in a setup with some integer number of pieces in the $F_1$ and $F_{m-1}$ columns.\pvs
    Since the Zeckendorf Game is playable and always results in the Zeckendorf decomposition of $n$, it is sufficient to show that any Black Hole Zeckendorf game with $n$ pieces and a black hole on $F_m$ reduces to a board such that the weighted sum of pieces is $n \pmod {F_m}$, as that is equivalent to a board in the original Zeckendorf game. From here, the game proceeds as the general Zeckendorf game does, so will reduce to the Zeckendorf decomposition of $n \pmod {F_m}$. \pvs
    Both the add and merge options combine two pieces into one piece in a column corresponding to a greater $F_m$, therefore closer to the black hole. The split option moves one piece towards the black hole and one piece away from it. Since there are a finite number of pieces, all moves shift at least the same amount of pieces towards the black hole as they do away from it, meaning that pieces must eventually be placed into the black hole if the weighted sum of pieces is greater than $F_m$. Every time a piece is placed in the black hole, it decreases the value of the board by $F_m$. Thus, the game must eventually reduced to a board such that the weighted sum of all pieces is $n \pmod{F_m}$.
\end{proof}
Note that the weighted sum all pieces on the board is a non-increasing monovariant, as pieces must eventually be placed into the black hole and every time a piece is placed in the black hole, it reduces the value of the board by $F_m$. This is different from the original Zeckendorf game, in which the number of pieces on the board is a non-increasing monovariant, but the weighted sum of the board itself is constant. \pvs
We consider the games with black holes on $F_2$, $F_3$, and $F_4$. For the game with a black hole on $F_2$, possible outcomes are $(0)$ and $(1)$. For the game with a black hole on $F_3$, possible outcomes are $(0,0)$, $(1,0)$, and $(0,1)$. For the game with a black hole on $F_4$, possible outcomes are $(0,0,0)$, $(1,0,0)$, $(0,1,0)$, $(0,0,1)$ and $(1,0,1)$.\pvs
To play through games, we draw game trees, denoting the player to place the setup with bold blue text, and the player to move next with red text. We label each edge with the corresponding move, labeling as follows:
\begin{enumerate}
    \item M for merging,
    \item A\tsub{1} for adding from columns $F_1$ and $F_2$,
    \item A\tsub{2} for adding from columns $F_2$ and $F_3$,
    \item S\tsub{2} for splitting from column $F_2$, and
    \item S\tsub{3} for splitting from column $F_3$.
\end{enumerate} \pvs We organize our trees so that options $1$ through $5$ are considered from left to right. Note that not all moves are always possible. In the game with a black hole on $F_2$, the only possible move is to merge, and in the game with a black hole on $F_3$, only the moves $M$, $A_1$ and $S_1$ are possible. For game states that we reference later in the proof, we label them by the color that placed it, the round it was placed, and which move was used to place it, noting both if two moves can be used to reach that state. Again, we assume that Red is the next player to move from a given position. See Figure \ref{fig:exgame}, as an example of the possible game moves from  the setup $(a,b,c)$.\pvs
\begin{figure}
    \begin{center}
        \begin{tikzpicture}
            \node (ST1) at (0,0){$\ST{(a,b,c)}$};
	           \node [label={[label distance=1mm]270:R.1.M}](NS1i) at (-6,-1.5){$\NS{(a-2,b+1,c)}$};
		        \node [label={[label distance=1mm]270:R.1.A\tsub{1}}](NS1ii) at (-3,-1.5){$\NS{(a-1,b-1,c+1)}$};
                \node [label={[label distance=1mm]270:R.1.A\tsub{2}}](NS1iii) at (0,-1.5){$\NS{(a,b-1,c-1)}$};
                \node [label={[label distance=1mm]270:R.1.S\tsub{2}}](NS1iv) at (3,-1.5){$\NS{(a+1,b-2,c+1)}$};
                \node [label={[label distance=1mm]270:R.1.S\tsub{3}}](NS1v) at (6,-1.5){$\NS{(a+1,b,c-2)}$};
            \draw  (ST1) -- node[left=4mm] {M}(NS1i);
            \draw  (ST1) -- node[left=2mm] {A\tsub{1}}(NS1ii);
            \draw  (ST1) -- node[right] {A\tsub{2}}(NS1iii);
            \draw  (ST1) -- node[right=2mm] {S\tsub{2}}(NS1iv);
            \draw  (ST1) -- node[right=4mm] {S\tsub{3}}(NS1v);
        \end{tikzpicture}
    \end{center}
    \caption{Example Game Tree for a Setup $(a,b,c)$,}
    \label{fig:exgame}
\end{figure}
\section{Game with a Black hole on \texorpdfstring{$F_1$}{F1} or \texorpdfstring{$F_2$}{F2}}
We first consider the $F_1$ and $F_2$ Black Hole Zeckendorf games.\pvs
As we define it, the $F_1$ Black Hole Zeckendorf game is not possible, as every piece is immediately trapped in the black hole. There is also no Empty Board Black Hole Zeckendorf game to play, as by definition, players cannot place on the black hole.\pvs
In an $F_2$ Black Hole Zeckendorf game, the only possible move is to merge pieces into the black hole and there is only one column to place in for the Empty Board game, meaning it is deterministic. We outline the winners for the $F_2$ Black Hole Zeckendorf game and the Empty Board $F_2$ Black Hole Zeckendorf game below. \pvs
\begin{thm}
    Let $(a)$ be a setup for an $F_2$ Black Hole Zeckendorf game. This is a $P$ position for $a\equiv 0,1 \pmod 4$ and an $N$ position for $a \equiv 2,3 \pmod 4$. For the Empty Board $F_2$ Black Hole Zeckendorf game with $n$ pieces, Player $1$ wins for all $n\equiv 1,2 \pmod 4$ and Player $2$ wins for all $n \equiv 0,3 \pmod 4$.
\end{thm}
\begin{proof}
    We proceed by induction on $a$. As base cases, see that there are no moves from $(0)$ or $(1)$, so it is trivially a $P$ position. Then, for induction's sake, assume $(a)$ is a $P$ position for some $a \equiv 0,1 \pmod 4$ and consider $a +4  \equiv 0,1 \pmod 4$. The following move must be $(a+2)$, and the move after must be $(a)\equiv 0,1 \pmod 4$, which is a $P$ position by assumption. Thus, $(a)$ is a $P$ position for all $a \equiv 0,1 \pmod 4$.
    It follows that $(a)$ is an $N$ position for all $a \equiv 2,3 \pmod 4$, as the following player places some $(a-2) \equiv 0,1 \pmod 4$, which we showed is a $P$ position.\pvs
    For the Empty Board $F_2$ Black Hole Zeckendorf game, players can only place pieces in the $F_1$ column, which has weight 1. Player $2$ concludes the empty board phase when $n$ is even, so they win if they set a $P$ position, and lose if they set an $N$ position.  Therefore, Player $2$ wins for all $n\equiv 0 \pmod 4$, and Player $1$ wins for all $n \equiv 2 \pmod 4$. \pvs
    When $n$ is odd, Player $1$ concludes the empty board phase, and likewise, they win if they set a $P$ position and lose if they set an $N$ position. Therefore, Player $1$ wins for all $n \equiv 1 \pmod 4$, and Player $2$ wins for all $n \equiv 3 \pmod 4$.
\end{proof}
\section{Game with a Black Hole on \texorpdfstring{$F_3$}{F3}}
We now consider the game with a black hole on $F_3=3$. Here players can choose how they move, making the game more interesting.
\subsection{Single Column Winning Board Setups}
We first consider which setups are $P$ positions. We start by considering all pieces in one column.
\begin{thm}\label{thm:Swin}
Let $(a,b)$ be a setup for an $F_3$ Black Hole Zeckendorf game. When $a\equiv 0,1 \pmod 3$ and $b=0$, this is a $P$ position and likewise this is a $P$ position when $a=0$ and $b\equiv 0,1 \pmod 3$.
\end{thm}
\begin{proof}
    For games with $a\equiv 0,1 \pmod 3$ and $b=0$, we proceed by induction on $\a$. If the game starts with $0$ or $1$ pieces, in either the $F_1$ or $F_2$ column, then $(a,b)$ is trivially a $P$ position.\pvs
    As our inductive hypothesis, assume that $(3\a+k_1,0)$ is a $P$ position for a fixed $\a$ with $k_1=0,1$. Then, consider the corresponding game trees below, with $k_1=0$ on the left and $k_1=1$ on the right.  \pvs
    \begin{center}
    \begin{tikzpicture}
\node (AST1) at (-3,0){$\ST{(3(\a+1),0)}$};
	\node (ANS1i) at (-3,-1.5){$\NS{(3\a+1,1)}$};
		\node (AST2i) at (-3,-3){$\ST{(3\a,0)}$};
\node (BST1) at (3,0){$\ST{(3(\a+1)+1,0)}$};
	\node (BNS1i) at (3,-1.5){$\NS{(3\a+2,1)}$};
		\node (BST2i) at (3,-3){$\ST{(3\a+1,0)}$};
            \draw  (AST1) -- node[right] {M}(ANS1i);
            \draw  (ANS1i) -- node[right] {A\tsub{1}}(AST2i);
            \draw  (BST1) -- node[right] {M}(BNS1i);
            \draw  (BNS1i) -- node[right] {A\tsub{1}}(BST2i);
        \end{tikzpicture}
\end{center}\pvs
    In both trees, it is possible to reduce down to a setup that wins by the inductive hypothesis, proving that it is indeed a $P$ position.\pvs
    For games with $b\equiv 0,1 \pmod 3$ and $a=0$, we proceed by induction on $\b$. Assume as an inductive hypothesis that $(0,3\b+k_2)$ is a $P$ position for a fixed $\b$ with $k_2=0,1$. Consider the corresponding game trees below, with $k_2=0$ on the left and $k_2=1$ on the right. \pvs
    \begin{center}
    \begin{tikzpicture}
\node (AST1) at (-3,0){$\ST{(0,3(\b+1))}$};
	\node (ANS1i) at (-3,-1.5){$\NS{(1,3\b+1)}$};
		\node (AST2i) at (-3,-3){$\ST{(0,3\b)}$};
\node (BST1) at (3,0){$\ST{(0,3(\b+1)+1)}$};
	\node (BNS1i) at (3,-1.5){$\NS{(1,3\b+2)}$};
		\node (BST2i) at (3,-3){$\ST{(0,3\b+1)}$};

            \draw  (AST1) -- node[right] {S\tsub{2}}(ANS1i);
            \draw  (ANS1i) -- node[right] {A\tsub{1}}(AST2i);
            \draw  (BST1) -- node[right] {S\tsub{2}}(BNS1i);
            \draw  (BNS1i) -- node[right] {A\tsub{1}}(BST2i);
        \end{tikzpicture}
\end{center}\pvs
Again, it is possible to reduce to a setup that wins by the inductive hypothesis, proving the claim.\pvs

Therefore, for an $F_3$ Black Hole Zeckendorf game, $(a,b)$ is a $P$ position for all board setups with $a \text{ or } b \equiv 0,1 \pmod 3$, and the other equal to $0$.
\end{proof}\pvs
\begin{thm}\label{thm:NSwin}
Let $(a,b)$ be a setup for an $F_3$ Black Hole Zeckendorf game. If $a \equiv 2 \pmod 3$ and $b=0$ or starts with $a=0$ and $b\equiv 2 \pmod 3$, then $(a,b)$ is an $N$ position.
\end{thm}
\begin{proof}
    We proceed by induction. If the game starts with $2$ pieces in the $F_1$ column and none in the other, the next player to move can merge to place $(0,1)$ which wins. If the game starts with $2$ pieces in the $F_2$ column, and none in the $F_1$ column, the next player to move can split to place $(1,0)$ which wins. \pvs
    For induction's sake, suppose the setups $(3\a+2,0)$ and $(0,3\b+2)$ are $N$ positions. Since the only possible moves from here are $(3\a,1)$ and $(1,3\b)$ respectively, this assumption also implies that $(3\a,1)$ and $(1,3\b)$ are $P$ positions. Then, consider the game trees below, inducting on $\a$ in the left tree and $\b$ in the right tree. \begin{center}
    \begin{tikzpicture}
\node (AST1) at (-4,0){$\ST{(3(\a+1)+2,0)}$};
	\node (ANS1i) at (-4,-1.5){$\NS{(3(\a+1),1)}$};
		\node (AST2i) at (-2,-3){$\ST{(3\a+2,0)}$};
        \node (AST2ii) at (-6,-3){$\ST{(3\a+1,2)}$};
            \node (ANS2i) at (-4,-4.5){$\NS{(3\a,1)}$};
\node (BST1) at (4,0){$\ST{(0,3(\b+1)+2)}$};
	\node (BNS1i) at (4,-1.5){$\NS{(1,3(\b+1))}$};
		\node (BST2i) at (2,-3){$\ST{(0,3\b+2)}$};
        \node (BST2ii) at (6,-3){$\ST{(2,3\b+1)}$};
            \node (BNS2i) at (4,-4.5){$\NS{(1,3\b)}$};
            \draw  (AST1) -- node[right] {M}(ANS1i);
            \draw  (ANS1i) -- node[right=3mm] {A\tsub{1}}(AST2i);
            \draw  (ANS1i) -- node[left=3mm] {M}(AST2ii);
            \draw  (AST2i) -- node[right=3mm] {M}(ANS2i);
            \draw  (AST2ii) -- node[left=3mm] {A\tsub{1}}(ANS2i);

            \draw  (BST1) -- node[right] {S\tsub{2}}(BNS1i);
            \draw  (BNS1i) -- node[left=3mm] {A\tsub{1}}(BST2i);
            \draw  (BNS1i) -- node[right=3mm] {S\tsub{2}}(BST2ii);
            \draw  (BST2ii) -- node[right=3mm] {A\tsub{1}}(BNS2i);
            \draw  (BST2i) -- node[left=3mm] {S\tsub{2}}(BNS2i);
        \end{tikzpicture}
\end{center}\pvs
Since we assumed $(3\a,1)$ and $(1,3\b)$ are $P$ positions, and showed that it is possible to reduce there from $(3(\a+1),1)$ and $(1,3(\b+1))$ respectively, it follows that boards of the form $(3\a+2,0)$ or $(0,3\b+2)$ are both $N$ positions
\end{proof}\pvs
\begin{corollary}\label{cor:NSwin}
    Let $(a,b)$ be a setup for an $F_3$ Black Hole Zeckendorf game. $(3\a,1)$ and $(1,3\b)$ are both $P$ positions.
\end{corollary}
\begin{proof}
    As base cases $(1,0)$ and $(0,1)$ both win trivially. As shown in the proof of Theorem \ref{thm:NSwin}, any player who places $(3\a,1)$ or $(1,3\b)$ can reduce to $(3(\a-1),1)$ and $(1,3(\b-1))$ respectively, so $(3\a,1)$ or $(1,3\b)$ win by induction.
\end{proof}
\subsection{General Winning Setups}
We now consider general winning setups of the form $(a,b)$. We show any case can be reduced modulo 3, and since both players have winning strategies in the lower cases, both players have winning strategies for higher cases. \pvs
\subsubsection{$P$ Positions}
\begin{thm}\label{thm:winequiv0}
    Let $(a,b)$ be a setup for an $F_3$ Black Hole Zeckendorf game. If $a \equiv b \equiv 0 \pmod 3$, then $(a,b)$ is a $P$ position.
\end{thm}
\begin{proof}
    We proceed by induction on $\b$. First, with the base case of $\b=0$, $(3\a,0)$ is a $P$ position by Theorem \ref{thm:Swin}. Then, let our inductive hypothesis be that  $(3\a,3\b)$ is a $P$ position for all $\a$ and some fixed $\b$. Then, consider the following game tree on $(3\a,3(\b+1))$.
    
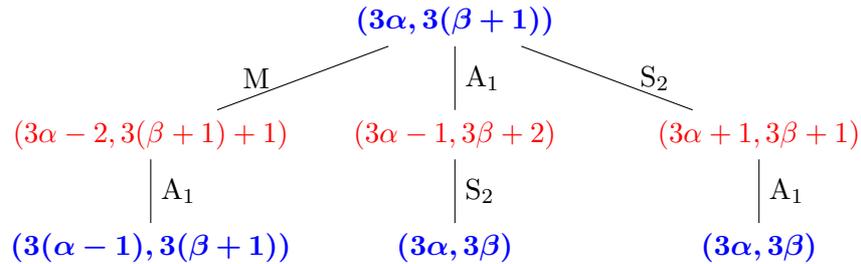
\begin{center}
    \begin{tikzpicture}
            \node (ST1) at (0,0){$\ST{(3\a,3(\b+1))}$};
                \node (NS1i) at (0,-1.5){$\NS{(3\a-1,3\b+2)}$};
                 	\node (ST2i) at (0,-3){$\ST{(3\a,3\b)}$};
                 \node (NS1ii) at (-4,-1.5){$\NS{(3\a-2,3(\b+1)+1)}$};
                  	\node (ST2ii) at (-4,-3){$\ST{(3(\a-1),3(\b+1))}$};
                  \node (NS1iii) at (4,-1.5){$\NS{(3\a+1,3\b+1)}$};
                  	\node (ST2iii) at (4,-3){$\ST{(3\a,3\b)}$};
            \draw  (ST1) -- node[right] {A\tsub{1}}(NS1i);
            \draw  (ST1) -- node[left=3mm] {M}(NS1ii);
            \draw  (ST1) -- node[right=3mm] {S\tsub{2}}(NS1iii);
            \draw  (NS1i) -- node[right] {S\tsub{2}}(ST2i);
           \draw  (NS1ii) -- node[right] {A\tsub{1}}(ST2ii);
           \draw  (NS1iii) -- node[right] {A\tsub{1}}(ST2iii);
       \end{tikzpicture}
    \end{center}\pvs
    In the right and center columns, it is possible to reduce to a game of the form $(3\a,3\b)$, thereby which is a $P$ position by the inductive hypothesis. The left column is not immediately a $P$ position by induction but the game options are the same as before, so it can be reduced to a $P$ position anytime the next player adds or splits. If the next player always chooses to merge, then the game eventually reduces to  $(0,3(\b+1))$ which is a $P$ position by Theorem \ref{thm:Swin}. Thus, $(a,b)$ is a $P$ position for the $F_3$ Black Hole Zeckendorf game if $a \equiv b \equiv 0 \pmod 3$.
\end{proof}\pvs
\begin{thm}\label{thm:winequiv10}
    Let $(a,b)$ be a setup for an $F_3$ Black Hole Zeckendorf Game. If $a \equiv 1 \pmod 3$ and $b \equiv 0 \pmod 3$, or vice versa, $(a,b)$ is a $P$ position.
\end{thm}
\begin{proof}
    We proceed by induction, first considering the game  $(3\a+1,3\b)$. First, fixing $\a=0$, we have that $(1,3\b)$ is a $P$ position by Corollary \ref{cor:NSwin}. Then, for our inductive hypothesis, suppose any setup of the form $(3\a+1,3\b)$ is a $P$ position for fixed $\a$ and all $\b$. Then, consider the following game tree on$(3(\a+1)+1,3\b)$.\pvs
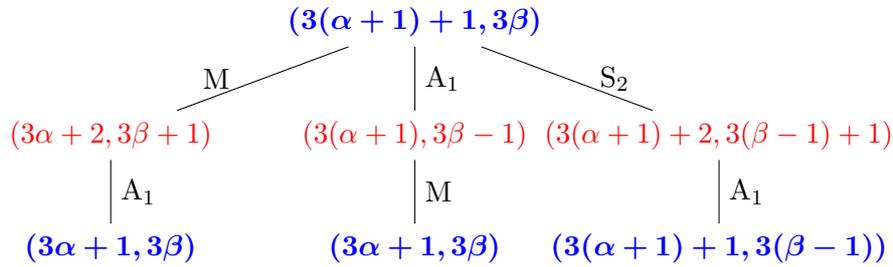
     \begin{center}
         \begin{tikzpicture}
            \node (ST1) at (0,0){$\ST{(3(\a+1)+1,3\b)}$};
                \node (NS1ii) at (-4,-1.5){$\NS{(3\a+2,3\b+1)}$};
                  	\node (ST2ii) at (-4,-3){$\ST{(3\a+1,3\b)}$};
                \node (NS1i) at (0,-1.5){$\NS{(3(\a+1),3\b-1)}$};
                 	\node (ST2i) at (0,-3){$\ST{(3\a+1,3\b)}$};
                  \node (NS1iii) at (4,-1.5){$\NS{(3(\a+1)+2,3(\b-1)+1)}$};
                  	\node (ST2iii) at (4,-3){$\ST{(3(\a+1)+1,3(\b-1))}$};
            \draw  (ST1) -- node[right] {A\tsub{1}}(NS1i);
            \draw  (ST1) -- node[left=3mm] {M}(NS1ii);
            \draw  (ST1) -- node[right=3mm] {S\tsub{2}}(NS1iii);
            \draw  (NS1i) -- node[right] {M}(ST2i);
           \draw  (NS1ii) -- node[right] {A\tsub{1}}(ST2ii);
           \draw  (NS1iii) -- node[right] {A\tsub{1}}(ST2iii);
       \end{tikzpicture}
    \end{center}\pvs
    In the left and center columns, it is possible to reduce to a game of the form $(3\a+1,3\b)$, which is a $P$ position by the inductive hypothesis. The right column is not immediately a $P$ position by inductive hypothesis, but the game options are the same as before so the will reduce to a $P$ position by inductive hypothesis any time the next player to move merges or adds. If the next player always splits, then the game eventually reduces to $(3(\a+1)+1,0)$ which is a $P$ position by Theorem \ref{thm:Swin}.\pvs
    Similarly, consider the initial setup $(3\a,3\b+1)$. Fixing $\b=0$, $(3\a,1)$ is a $P$ position by Corollary \ref{cor:NSwin}. Then, as our inductive hypothesis, we assume that $(3\a,3\b+1)$ is a $P$ position for some fixed $\b$ and all $\a$. We consider the game on $(3\a,3(\b+1)+1)$.\pvs
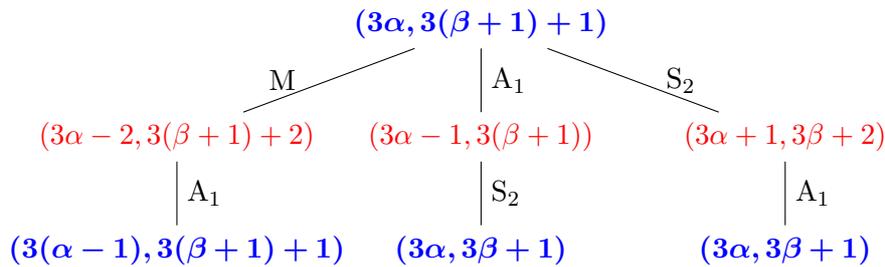
     \begin{center}
       \begin{tikzpicture}
            \node (ST1) at (0,0){$\ST{(3\a,3(\b+1)+1)}$};
                \node (NS1i) at (0,-1.5){$\NS{(3\a-1,3(\b+1))}$};
                 	\node (ST2i) at (0,-3){$\ST{(3\a,3\b+1)}$};
                 \node (NS1ii) at (-4,-1.5){$\NS{(3\a-2,3(\b+1)+2)}$};
                  	\node (ST2ii) at (-4,-3){$\ST{(3(\a-1),3(\b+1)+1)}$};
                  \node (NS1iii) at (4,-1.5){$\NS{(3\a+1,3\b+2)}$};
                  	\node (ST2iii) at (4,-3){$\ST{(3\a,3\b+1)}$};
            \draw  (ST1) -- node[right] {A\tsub{1}}(NS1i);
            \draw  (ST1) -- node[left=3mm] {M}(NS1ii);
            \draw  (ST1) -- node[right=3mm] {S\tsub{2}}(NS1iii);
            \draw  (NS1i) -- node[right] {S\tsub{2}}(ST2i);
           \draw  (NS1ii) -- node[right] {A\tsub{1}}(ST2ii);
           \draw  (NS1iii) -- node[right] {A\tsub{1}}(ST2iii);
       \end{tikzpicture}
    \end{center}\pvs
    The center and right columns are $P$ positions by inductive hypothesis. Again, $(3(\a-1),3(\b+1)+1)$ is not immediately a $P$ position by induction. However, it is either a $P$ position by induction once the next player chooses to add or split, or once the game is reduced to $(0,3(\b+1)+1)$, as a result of Theorem \ref{thm:Swin}.\pvs
    Hence, $(a,b)$ is a $P$ position if $a \equiv 1 \pmod 3$ and $b \equiv 0 \pmod 3$ or vice versa.
\end{proof}\pvs
\subsubsection{$N$ Positions}
\begin{thm}\label{thm:SetupNSwins}
   Let $(a,b)$ be a setup for an $F_3$ Black Hole Zeckendorf game. The possibilities listed below are all $N$ positions.
   \begin{enumerate}
       \item $a \equiv b \equiv 1 \pmod 3$,
       \item $a \equiv b \equiv 2 \pmod 3$,
       \item $a \equiv 2 \pmod 3$ and $b \equiv 1 \pmod 3$,
       \item $a\equiv 1 \pmod 3$ and $b\equiv 2 \pmod 3$,
       \item $a \equiv 0 \pmod 3$ and $b \equiv 2\pmod 3$,
       \item $a \equiv 2 \pmod 3$ and $b \equiv 0 \pmod 3$.
   \end{enumerate}
\end{thm}
\begin{proof}
     For each statement, we show that each position leads directly to a position that we showed is $P$ above.\pvs
    First, suppose the board is set as $(3\a+1,3\b+1)$. Then, the next player can add from columns $F_1$ and $F_2$ to place $(3\a,3\b)$ which is a $P$ position as shown in Theorem \ref{thm:winequiv0}.\pvs
    Similarly, suppose the board is set as $(3\a+2,3\b+2)$. Then the next player can merge from $F_1$ to place the board as $(3\a,3(\b+1))$ which is a $P$ position by Theorem \ref{thm:winequiv0} again.\pvs
    Then, suppose the board is set as $(3\a+2,3\b+1)$ or $(3\a+1,3\b+2)$. The next player can add from columns $F_1$ and $F_2$ to place $(3\a+1,3\b)$ and $(3\a,3\b+1)$ respectively, which by Theorem \ref{thm:winequiv10} are both $P$ positions.\pvs
    Then, suppose the board is set as $(3\a,3\b+2)$. The next player can split from column $F_2$, to place the board as $(3\a+1,3\b)$, which is a $P$ position by Theorem \ref{thm:winequiv10}.\pvs
    Lastly, suppose the board is set as $(3\a+2, 3\b)$. The next player can merge from column $F_1$ to place the board as ${(3\a,3\b+1)}$, which is a $P$ position by Theorem \ref{thm:winequiv10}.
\end{proof}\pvs
\subsection{Empty Board Game}
We now continue to the Empty Board $F_3$ Black Hole Zeckendorf game and determine which player has a winning strategy for any given $n \in \Zp$. We find that Player $1$ has a constructive winning strategy for $n\equiv 1,2,3,6,8\pmod 9$ and Player $2$ has a constructive winning strategy for $n \equiv 0,4,5,7 \pmod 9$.
\begin{thm}\label{thm: evensplit}
   Let $(0,0)$ be the beginning board for an Empty Board $F_3$ Black Hole Zeckendorf game with $n \in \Zp$ pieces. Players can force certain game setups as outlined below.
   \begin{enumerate}
       \item For any $n \equiv \pm 3 \pmod 9$, Player 1 can force the game into a setup $(n/3,n/3)$.
       \item For any $n \equiv 0 \pmod 9$, Player 2 can force the game into a setup $(n/3,n/3)$.
       \item For any $n \equiv 1 \pmod 9$, Player 1 can force the game into a setup $((n-1)/3+1,(n-1)/3)$.
       \item For any $n \equiv 4,7 \pmod 9$, Player 2 can force the game into a setup $((n-1)/3+1,(n-1)/3)$.
       \item For any $n \equiv 2 \pmod 9$, Player 1 can force the game into a setup $((n-2)/3,(n-2)/3+1)$.
       \item For any $n \equiv 8 \pmod 9$, Player 1 can force the game into a setup $((n-2)/3+2,(n-2)/3)$.
       \item  For any $n \equiv 5 \pmod 9 $, Player $2$ can force the game into a setup that is either\\
       $((n-2)/3+2,(n-2)/3)$ or $((n-2)/3,(n-2)/3+1)$.
   \end{enumerate}.
\end{thm}
\begin{proof}
    First, consider $n \equiv \pm 3 \pmod 9$. Player 1 should place their first piece in the $F_2$ column. Then, they should act opposite of Player $2$, until there is one piece left, which  Player $2$ will be forced to place in the $F_1$ column, setting up the board as $(n/3,n/3)$. \pvs
    If $n \equiv 0 \pmod 9$, Player 2 can force the setup $(n/3,n/3)$ by placing opposite Player 1, resulting in Player $2$ placing down $(n/3, n/3)$.\pvs
    Then, consider $n \equiv 1 \pmod 9$. Player $1$ should place their first piece in the $F_1$ column. For every round following, Player $1$ should place in the opposite column of Player $2$. Since $n\equiv 1 \pmod 3$, this means that Player $1$ will place the last piece, setting the board as $((n-1)/3+1,(n-1)/3)$. \pvs
    For $n \equiv 4,7 \pmod 9$, Player $2$ should place opposite of Player $1$. Again, this means that Player $1$ will place the last piece, setting the board as $((n-1)/3+1,(n-1)/3)$.\pvs
     Next, consider $n \equiv 2 \pmod 9$. Player $1$ should place their first piece in the $F_2$ column, and then play opposite Player $2$ until Player $1$ eventually sets the board as $((n-2)/3, (n-2)/3+1)$.\pvs
    Next, consider $n \equiv 8 \pmod 9$. Player $1$ should place their first piece in the $F_1$ column, and then play opposite Player $2$. Then, every time after Player $1$ places down, the number of pieces left to place will be some $p \equiv 1 \pmod 3$, meaning Player $2$ will eventually be forced to set $((n-2)/3+2, (n-2)/3)$.\pvs
    Lastly, for $n \equiv 5 \pmod 9$, Player $2$ should play opposite all of Player $1$'s moves, so that Player $2$ eventually place the board $((n-2)/3,(n-2)/3)$. This forces Player $1$ to either place $((n-2)/3,(n-2)/3+1)$ or $((n-2)/3+1,(n-2)/3)$, the latter allowing Player $2$ to place the setup $((n-2)/3+2,(n-2)/3)$.
\end{proof}\pvs
 Note that these are the only setups that can be forced, as Player $1$ has at most one more than half of the moves, while Player $2$ has at most half of the moves. In the world of combinatorial game theory, this is often referred to as the``Tweedledum and Tweedledee" strategy, where one player always plays opposite the other; it can be used in games like Hackenbush \cite{BCG} and Clobber \cite{ANW}.
\begin{thm}\label{thm: f3wins}
    Let $(0,0)$ be the beginning board for an Empty Board $F_3$ Black Hole Zeckendorf Game with $n$ pieces. Player $1$ has a constructive strategy for winning for any $n \equiv 1,2,3,6,8 \pmod 9$. Player $2$ has a constructive strategy for winning for any $n \equiv 0,4,5,7 \pmod 9$.
\end{thm}
\begin{proof}
    We first consider when $n \equiv 3,6\pmod 9$. By Theorem \ref{thm: evensplit}, Player $1$ can force Player $2$ to set the board as $(n/3,n/3)$. Then, $n\equiv 0\pmod 3$ but $n/3 \not \equiv 0 \pmod 3$. Since $a,b \not \equiv 0 \pmod 3$ this is an $N$ position by Theorem \ref{thm:SetupNSwins}, so Player $1$ has a constructive solution.\pvs
    Then, consider $n\equiv 0 \pmod 9$, where Player $2$ can force the game so that they place the setup $(n/3,n/3)$ by Theorem \ref{thm: evensplit}. Since $n\equiv 0 \pmod 9$ implies $ n/3 \equiv 0 \pmod 3$, this is a $P$ position by Theorem \ref{thm:winequiv0}, so Player $2$ has a constructive winning strategy.\pvs
    Next, we consider when $n \equiv 1 \pmod 9$, meaning Player $1$ sets the board as $((n-1)/3+1,(n-1)/3)$ by Theorem \ref{thm: evensplit}. Since $(n-1)/3+1 \equiv 1 \pmod 3$ and $(n-1)/3 \equiv 0 \pmod 3$ this is a $P$ position by Theorem \ref{thm:winequiv10}, so Player 1 has a constructive strategy for winning. \pvs
    Then, suppose $n \equiv 4,7 \pmod 9$, meaning Player $2$ can force Player $1$ to place the setup $((n-1)/3+1,(n-1)/3)$ by Theorem \ref{thm: evensplit}. For $n \equiv 4 \pmod 9$, $(n-1)/3+1 \equiv 2 \pmod 3$ and $(n-1)/3 \equiv 1 \pmod 3$, so by Theorem \ref{thm:SetupNSwins}, this is an $N$ position so Player 2 has a winning strategy. Similarly, if $n \equiv 7 \pmod 9$, then $(n-1)/3+1 \equiv 0 \pmod 3$ and $(n-1)/3 \equiv 2 \pmod 3$ so Player 2 has a winning strategy by Theorem \ref{thm:SetupNSwins}.\pvs
     Then, we consider when $n \equiv 2 \pmod 9$, meaning that Player 1 is able to place the setup $((n-2)/3,(n-2)/3+1)$ by Theorem \ref{thm: evensplit}. Since $(n-2)/3+1 \equiv 1 \pmod 3$ and $(n-2)/3 \equiv 0 \pmod 3$, this is a $P$ position by Theorem $\ref{thm:winequiv10}$ so Player 1 wins. \pvs
     If $n \equiv 8 \pmod 9$, then Player $1$ can force Player $2$ to place the setup $((n-2)/3+2,(n-2)/3)$ by Theorem \ref{thm: evensplit}. Since $(n-2)/3+2 \equiv 1 \pmod 3$ and $(n-2)/3 \equiv 2 \pmod 3$, then Player 1 wins by Theorem \ref{thm:SetupNSwins}. \pvs
    If $n \equiv 5 \pmod 9$, then by by Theorem \ref{thm: evensplit}, Player $2$ can force the game so that either they place the setup $((n-2)/3+2,(n-2)/3)$ or Player $1$ places the setup $((n-2)/3,(n-2)/3+1)$. We note that $(n-2)/3 \equiv 1 \pmod 3$, $(n-2)/3+1 \equiv 2 \pmod 3$ and $(n-2)/3+2 \equiv 0 \pmod 3$. It follows that Player $2$ wins the setup $((n-2)/3,(n-2)/3+1)$ by Theorem \ref{thm:SetupNSwins}, since it is an $N$ position. Player $2$ also wins the setup $((n-2)/3+2,(n-2)/3)$ as by Theorem \ref{thm:winequiv10} it is a $P$ position.
\end{proof} \pvs
This concludes our analysis of the $F_3$ Black Hole Zeckendorf game, both with and without the Empty Board phase of the game.
\section{Game with a Black Hole on \texorpdfstring{$F_4$}{F4}}
We now expand to the game with a black hole on $F_4=5$ with the setup $(a,b,c)$. This game becomes more interesting for a variety of reasons. There are more possible moves for each player and less symmetry, leading to a solution that is constructive but more intricate. Here, the winner is not solely based on the equivalence classes of $a$ and $c$ but also their value in relation to each other. \pvs
\subsection{Single Column Winning Board Setups}
We first consider winning strategies when all pieces are in either in column $F_1$ or column $F_3$. We do not consider when all are in column $F_2$ for two reasons. First and foremost, we define the Empty Board game so that players are only able to place in the outermost columns, as this greatly simplifies players abilities to force certain setups. Second, a significant strategy in the solution is forcing all pieces to the outer columns, to prevent the other player from using the ``Add'' move, thereby limiting their options. This strategy clearly fails when all pieces are in the $F_2$ column, making analysis of the game challenging. A complete analysis of the $F_4$ Black Hole Zeckendorf game on $(a,b,c)$ such that $b \neq 0$ would prove interesting in future work, as would analyzing the Empty Board game without restrictions on placing in the $F_2$ column.\pvs
\begin{thm}\label{thrm: allina} Let $(a,0,0)$ be a game state for an $F_4$ Black Hole Zeckendorf game. For any $a\neq 2$, $(a,0,0)$ is a $P$ position.
\end{thm}
\begin{proof}
    We proceed by induction on $a$. See in Appendix \ref{bcaseallina} that this is a $P$ position for all base cases
    $a=1,3,4,5,7$. Then, for our inductive hypothesis suppose the starting position $(a,0,0)$ wins and the consider following game tree.
   
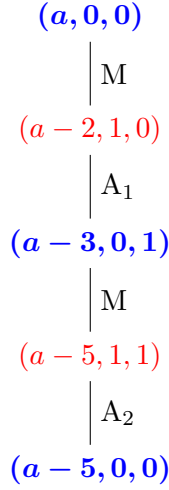
\begin{center}
        \begin{tikzpicture}
            \node (ST1) at (0,0){$\ST{(a+5,0,0)}$};
                \node (NS1) at (0,-1.5){$\NS{(a+3,1,0)}$};
                    \node (ST2) at (0,-3){$\ST{(a+2,0,1)}$};
                        \node (NS2) at (0,-4.5){$\NS{(a,1,1)}$};
                            \node (ST3i) at (0,-6){$\ST{(a,0,0)}$};
            \draw  (ST1) -- node[right] {M}(NS1);
            \draw  (NS1) -- node[right] {A\tsub{1}}(ST2);
            \draw  (ST2) -- node[right] {M}(NS2);
            \draw   (NS2) -- node[right] {A\tsub{2}}(ST3i);
        \end{tikzpicture}
    \end{center}\pvs
    Since any $(a,0,0)$ wins by inductive hypothesis, and as shown above, it is always possible to reduce $a$ modulo 5 until it reaches one of the base cases, then $(a,0,0)$ is a $P$ position for all $a\neq 2$.
\end{proof}\pvs
 \begin{corollary}\label{coralla} Let $(a,b,c)$ be a game state for an $F_4$ Black Hole Zeckendorf game. The position $(a,0,1)$ is a $P$ position for all $a\in \Znn$. The position $(a,0,2)$ is an $N$ position for all $a\neq 1 \in \Znn$. The positions $(a,0,3)$, $(a,1,1)$ and $(a,1,2)$ are $N$ positions for all $a\in \Znn$.
\end{corollary}
\begin{proof}
    In the game tree considered in Theorem \ref{thrm: allina} and the base cases in Appendix \ref{bcaseallina}, $(a,0,1)$ is a $P$ position since it reduces to $(a,0,0)$. Note in the case $(4,0,1)$ the tree above reduces to $(2,0,0)$ which is an $N$ position but the base case tree for $n=7$ as included in Appendix \ref{bcaseallina} reduces to $(0,1,0)$ which is a $P$ position. \pvs
    From $(a,0,2)$ and $(a,0,3)$, the next player can split in the $F_3$ column to place $(a+1,0,0)$ and $(a+1,0,1)$ respectively, which are $P$ positions as shown above. The only exception to this is when the board is set as $(1,0,2)$, given that $(2,0,0)$ is an $N$ position.  \pvs
    From $(a,1,1)$ and $(a,1,2)$ the next player can add from columns $F_2$ and $F_3$ to place $(a,0,0)$ and $(a,0,1)$ respectively, which are $P$ positions as shown above.
\end{proof}\pvs
 \begin{thm}\label{thrm: allinc} Let $(0,0,c)$ be an game state for an $F_3$ Black Hole Zeckendorf game. For any $c\neq 0,1,5 \in \Znn$, $(0,0,c)$ is an $N$ position. 
 \end{thm}
 \begin{proof}
     It is sufficient to prove that $(1,0,c)$ is a $P$ position, since the next player to move from $(0,0,c)$ can split in the third column to place $(1,0,c)$. We proceed by induction on $c$. See in Appendix \ref{bcaseallinc} that this is a $P$ position for all base cases ${c=0,1,2,4,8}$ while it is an $N$ position for $c=3$. Then, consider the following game tree from $(1,0,c)$
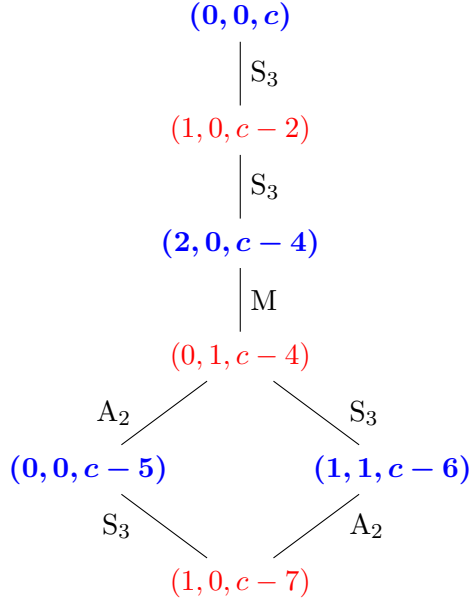
\begin{center}
        \begin{tikzpicture}
                \node (NS1) at (0,-1.5){$\ST{(1,0,c+5)}$};
                    \node (ST2) at (0,-3){$\NS{(2,0,c+3)}$};
                        \node (NS2) at (0,-4.5){$\ST{(0,1,c+3)}$};
                            \node (ST3i) at (-2,-6){$\NS{(0,0,c+2)}$};
                            \node (ST3ii) at (2,-6){$\NS{(1,1,c+1)}$};
                                \node (NS3i) at (0,-7.5){$\ST{(1,0,c)}$};

            \draw  (NS1) -- node[right] {S\tsub{3}}(ST2);
            \draw  (ST2) -- node[right] {M}(NS2);
            \draw   (NS2) -- node[left=3mm] {A\tsub{2}}(ST3i);
            \draw   (NS2) -- node[right=3mm] {S\tsub{3}}(ST3ii);
            \draw   (ST3i) -- node[left=3mm] {S\tsub{3}}(NS3i);
            \draw   (ST3ii) -- node[right=3mm] {A\tsub{2}}(NS3i);
        \end{tikzpicture}
    \end{center}\pvs
    Since $(1,0,c)$ wins by inductive hypothesis and as shown above it is always possible to reduce modulo 5 until the game reaches a base case, then $(1,0,c)$ is a $P$ position for all $c \neq 3$, which implies $(0,0,c)$ is an $N$ position for all $c\neq 0,1,5$.
 \end{proof}\pvs
 \begin{corollary}\label{corallc} Let $(a,b,c)$ be a game state for an $F_4$ Black Hole Zeckendorf game. The position $(1,0,c)$ is a $P$ position for all $c\neq 3 \in \Znn$ and the position $(0,1,c)$ is a $P$ position for all $c \neq 1,2,6 \in \Znn$. The position $(2,0,c)$ is an $N$ position for all $c \neq 1 \in \Znn$, and the positions $(1,1,c)$ and $(2,1,c)$ are $N$ position  for all $c \in \Znn$.
\end{corollary}
\begin{proof}
    Theorem \ref{thrm: allinc} proved that $(1,0,c)$ is a $P$ position for all $c \neq 3 \in \Znn$. Similarly, $(0,1,c)$ is a $P$ position for all $c\neq 1,2,6$ as it reduces to $(1,0,c)$ in the game tree of Theorem \ref{thrm: allinc}. This is true for all $c\neq 1,2,6$ since placing $(0,0,c-1)$ only wins for $c-1=0,1,5$. \pvs
    Next, see that $(2,0,c)$ is an $N$ position in the tree in Theorem \ref{thrm: allinc} and in the bases case trees in Appendix \ref{bcaseallinc}, so it is an $N$ position for $c\neq 1$. \pvs
    Lastly, see that $(1,1,c)$ and $(2,1,c)$ are $N$ position  for all $c \in \Znn$. The next player to move from $(2,1,c)$ can place $(1,0,c+1)$ which is a $P$ position for all $c\neq 2$. When $c=2$, this player can place $(2,0,1)$ which we also showed to be a $P$ position in Corollary \ref{coralla}. Placing $(1,1,c)$ is an $N$ position, given that the next player can place $(1,0,c-1)$ for $c\neq 4$, and $(0,0,5)$ for $c=4$, both of which are $P$ positions. 
\end{proof}\pvs
\subsection{General Winning Setups}
Our proof of general winning setups for a black hole on $F_4$ is more involved, as the minor exceptions in the cases above prevent us from creating a solution solely based on the equivalence classes of $a$ and $c$ for the setup $(a,0,c)$. To motivate this section, we remind the reader of Figure \ref{fig:winners(a,0,c)}, which outlines the winners for the board setup $(a,0,c)$ in an $F_4$ Black Hole Zeckendorf Game.\pvs
We construct a solution by providing a path to place a $P$ position. However, it is not possible to do this until we can verify that the $P$ positions as outlined in the table are in fact $P$ positions. We do this using a mixture of constructive and non-constructive methods.
\begin{lemma}\label{lemma:nck12}
  \textbf{(Non-constructive).} Let $(a,b,c)$ be a game state for an $F_4$ Black Hole Zeckendorf game. For all $\alpha,\gamma, k_1,k_3 \in \Znn$, such that $1 \leq k_1\leq 2$, and $0 \leq k_3\leq 3$, $(3\alpha+k_1,1,4\gamma+k_3)$ is an $N$ position.
\end{lemma}
\begin{proof}
    We proceed by a non-constructive proof. Recall that we assume that Red moves first from the position $(3\alpha+k_1,1,4\gamma+k_3)$. For contradiction's sake, suppose $(3\alpha+k_1,1,4\gamma+k_3)$ is a $P$ position when $k_1=1,2$. Then, consider the game tree in Figure \ref{fig:(3,1,4)}, where $r$ is the number of rounds in which both players move once. We note that Red has other possible moves, but we only consider moves relevant to the proof.  \pvs
    \begin{figure}
  \begin{center}
  \begin{tikzpicture}
 \node (ST1) at (0,0){$\ST{(3\alpha+k_1,1,4\gamma+k_3)}$};
 	\node (NS1i) at (-3,-1.5){$\NS{(3\alpha+k_1-1,0,4\gamma+k_3+1)}$};
		\node (ST2i) at (-3,-3){$\ST{(3\alpha+k_1-3,1,4\gamma+k_3+1)}$};
			\node (NS2i) at (-3,-4.5){$\NS{(3\alpha+k_1-4,0,4\gamma+k_3+2)}$};
				\node (ST3i) at (-3,-6){$\ST{(3\alpha+k_1-6,1,4\gamma+k_3+2)}$};
						\node (ST4i) at (-3,-7.5){. . . };
							\node (ST5i) at (-3,-9){$\ST{(3\alpha+k_1-3r,1,4\gamma+k_3+r)}$};
				\node (ST3ii) at (3,-6){$\ST{(3\alpha+k_1-3,0,4\gamma+k_3)}$};
			\node (NS2ii) at (3,-4.5){$\NS{(3\alpha+k_1-3,0,4\gamma+k_3)}$};
		 \node (ST2ii) at (3,-3){$\ST{(3\alpha+k_1,0,4\gamma+k_3-1)}$};
	 \node (NS1ii) at (3,-1.5){$\NS{(3\alpha+k_1,0,4\gamma+k_3-1)}$};
            \draw  (ST1) -- node[left=3mm] {A\tsub{1}}(NS1i);
            \draw  (ST1) -- node[right=3mm] {A\tsub{2}}(NS1ii);
            \draw  (NS1i) -- node[right] {M}(ST2i);
             \draw  (NS1i) -- node[right=6mm] {S\tsub{3}}(ST2ii);
            \draw  (ST2i) -- node[right] {A\tsub{1}}(NS2i);
             \draw  (ST2i) -- node[right=6mm] {A\tsub{2}}(NS2ii);
             \draw   (NS2i) -- node[right] {M}(ST3i);
             \draw   (NS2i) -- node[right=6mm] {S\tsub{3}}(ST3ii);
             \draw  (ST3i) -- (ST4i);
              \draw  (ST4i) -- (ST5i);
        \end{tikzpicture}
        \caption{Game Tree for the Game State $(3\a+k_1,1,4\g+k_3)$.}
    \label{fig:(3,1,4)}
  \end{center}\pvs
  \end{figure}
By assumption, $(3\a+k_1,1,4\g+k_3)$ is a $P$ position regardless of what the other player places. Suppose Red places $(3\alpha+k_1-1,0,4\gamma+1)$ for their first move. Then, Blue can place either $(3\alpha+k_1-3,1,4\gamma+1)$ or $(3\alpha+k_1,0,4\gamma-1)$. But as shown in the tree, Red had the opportunity to place $(3\alpha+k_1,0,4\gamma-1)$ in the round before; so by assumption, placing it is a losing move. It follows that in order to win, Blue must place $(3\alpha+k_1-3,1,4\gamma+1)$. As shown in the tree, Red has the same options as before, so if they place $(3\alpha+k_1-4,0,4\gamma+k_3+2)$, then by assumption, Blue must place $(3\alpha+k_1-6,1,4\gamma+k_3+2)$. \pvs
Then, the game eventually reduces down to Blue placing $(3\alpha+k_1-3r,1,4\gamma+k_3+r)$ after $r$ rounds of each player moving once. After the $\alpha$\textsuperscript{th} round, Blue will place $(k_1,1,4\gamma+k_3+\alpha)$. If $k_1=1$, then Red can add from columns $F_2$ and $F_3$ to place $(1,0,4\gamma+\alpha+k_3-1)$ which is a $P$ position by Corollary \ref{corallc} for all $4\gamma+\alpha+k_3-1\neq 3$. If $k_1=2$, then Red can add from columns $F_1 $ and $ F_2$ to place $(1,0,4\gamma+\alpha+k_3+1)$ which wins by Corollary \ref{corallc} for all $4\gamma+\alpha+k_3+1\neq 3$. Additionally, we show in Appendix \ref{basecasea1c}, that Red also has a winning strategy for the cases when $4\gamma+\alpha+k_3 \pm 1 =3$. Since Red moves first, this is a contradiction to the assumption that for $k_1=1,2$, $(3\alpha+k_1,1,4\gamma+k_3)$ is a $P$ position, thus it is instead an $N$ position.
\end{proof}\pvs
\begin{rek}
    For the sake of conserving space and avoiding repetition, from this point forward we assume both players are seeking optimal strategies. Then, when $k_1=1,2$, we omit any move of the form $(3\alpha+k_1,1,4\gamma+k_3)$, as it immediately loses.
\end{rek}\pvs
From here, we consider when $k_1=0$, and show that $(3\alpha,1,4\gamma+1)$ is an $N$ position for $\alpha\geq \gamma$ and $(3\alpha,1,4\gamma)$ is an $N$ position for $\alpha \geq \gamma + 1$. \pvs
\begin{lemma}\label{lemma:(3,0,4)}
    Let $(a,b,c)$ be a game state for an $F_4$ Black Hole Zeckendorf game. For all $\a,\g \in \Znn$ such that $\alpha \geq \gamma$, $(3\alpha,0,4\gamma)$ is a $P$ position.
\end{lemma}
\begin{proof}
We proceed by strong induction on $\g$, letting $\a$ be an arbitrary integer such that $\a \geq \g$. As a base case, note that $(3\a,0,0)$ is a $P$ position by Theorem \ref{thrm: allina}, since $3\a\neq 2$ for any integer $\a$.  Then, consider the game on $(3\alpha,0,4(\g+1))$ as played out in Figure \ref{fig:(3,0,4)}, with $\a \geq \g+1$, assuming as our inductive hypothesis that there exists a fixed $\g$ such that placing $(3\alpha,0,4x)$ wins for all $\alpha \geq x$ and $x \leq \g$. \pvs
\begin{figure}
    \centering
   \begin{tikzpicture}
 \node (NS1) at (0,0){$\ST{(3\a,0,4(\g+1))}$};
	 \node (ST1ii) at(0,-1.5) {$\NS{(3\a+1,0,4\g+2)}$};
	 	\node (NS2i) at (0,-3){$\ST{(3\a+2,0,4\g)}$};
 			\node (ST2ii) at (5,-4.5){$\NS{(3(\a+1),0,4(\g-1)+2)}$};
				\node (NS3ii) at (5,-6){$\ST{(3(\a+1)+1,0,4(\g-1))}$};
					\node (ST3iii) at (5,-7.5){$\NS{(3(\a+1)+2,0,4(\g-2)+2)}$};
						\node  [label={[label distance=1mm]270:B.4.S\tsub{3}}](NS4iii) at (5,-9){$\ST{(3(\a+2),0,4(\g-2))}$};
			\node (ST2i) at (-5,-4.5){$\NS{(3\a,1,4\g)}$};
				\node (NS3i) at (-5,-6){$\ST{(3(\a-1)+2,0,4\g+1)}$};
					\node (ST3i) at (0,-7.5){$\NS{(3\a,0,4(\g-1)+3)}$};
						\node (NS4i) at (0,-9){$\ST{(3\a+1,0,4(\g-1)+1)}$};
							\node (ST4i) at (0,-10.5){$\NS{(3\a+2,0,4(\g-2)+3)}$};
                                \node (NS5i) at (0,-12){$\ST{(3\a+3,0,4(\g-2)+1)}$};
                                \node (NS6i) at (0,-13.5){$\ST{. . .}$};
                                \node [label={[label distance=1mm]270:B.7.S\tsub{3}}](NS7i) at (0,-15){$\ST{(3\a+1+2r,0,4(\g-1-r)+1)}$};
					\node (ST3ii) at (-5,-7.5){$\NS{(3(\a-1),1,4\g+1)}$};
						\node [label={[label distance=1mm]270:B.4.A\tsub{2}}](NS4ii) at (-5,-9){$\ST{(3(\a-1),0,4\g)}$};
             \draw  (NS1) -- node[right] {S\tsub{3}}(ST1ii);
             \draw  (ST1ii) -- node[right] {S\tsub{3}}(NS2i);
             \draw  (NS2i) -- node[left=5mm] {M}(ST2i);
                 \draw  (ST2i) -- node[right] {A\tsub{1}}(NS3i);
                 \draw  (NS3i) -- node[right=5mm] {S\tsub{3}}(ST3i);
                 \draw  (ST3i) -- node[right] {S\tsub{3}}(NS4i);
                 \draw  (NS4i) -- node[right] {S\tsub{3}}(ST4i);
                 \draw  (ST4i) -- node[right] {S\tsub{3}}(NS5i);
                 \draw  (NS5i) -- node[right] {S\tsub{3}}(NS6i);
                 \draw  (NS6i) -- node[right] {S\tsub{3}}(NS7i);
            \draw  (NS2i) -- node[right=5mm] {S\tsub{3}}(ST2ii);
            \draw  (ST2ii) -- node[right] {S\tsub{3}}(NS3ii);
                \draw  (NS3i) -- node[right] {M}(ST3ii);
                \draw  (ST3ii) -- node[right] {A\tsub{2}}(NS4ii);
            \draw  (NS3ii) -- node[right] {S\tsub{3}}(ST3iii);
            \draw  (ST3iii) -- node[right] {S\tsub{3}}(NS4iii);
    \end{tikzpicture}
    \caption{Game Tree for the Game State $(3\a,0,4(\g+1))$.}
    \label{fig:(3,0,4)}
  \end{figure}
Note $\a \geq \g +1$ implies that $\a -1 \geq \g$ and $\a+2 \geq \g-2$ so B.4.A\tsub{2} and B.4.S\tsub{3} both win by inductive hypothesis. Some special consideration is needed here for the case when $\g = 1$, since B.4.S\tsub{3} $= (3(\a+2),0,-4)$ is not a valid game state. Here, B.3.S\tsub{3} $= (3(\a+1)+1,0,0)$ is a $P$ position by Theorem \ref{thrm: allina} because $3(\a+1)+1 \neq 2$ for $\a \in \Znn$. \pvs
For the center branch of the tree, if both players always choose to split once the tree is of the form $(3\a+1,0,4(\g-1)+1)$, then Blue always places a board of the form $(3\a+1+2r,0,4(\g-1-r)+1)$, where $r$ is a round of each player moving once. Then, after the $(\g-1)$\textsuperscript{st} round, Blue will place down $(3\a+1+2(\g-1),0,1)$ which is a $P$ position by Corollary \ref{coralla}. \pvs
However, it is necessary to consider that every time Blue places $(3\a+1+2r,0,4(\g-1-r)+1)$ with $2r \equiv 1 \pmod{3}$, Red has the opportunity to place $(3\a+2r-1,1,4(\g-1-r)+1)$ without immediately losing by Lemma \ref{lemma:nck12}, since here $3\a+2r-1 \equiv 0 \pmod{3}$. However, from here Blue can add from columns $F_2$ and $F_3$ to place $(3\a+2r-1,0,4(\g-1-r))$ which wins by inductive hypothesis, since $\a \geq \g-1-r$. An example of this can be seen in the left most column. \pvs

Therefore, $(3\a,0,4\g)$ is a $P$ position for all $\a\geq \g$.

\end{proof} \newpage
\begin{corollary}\label{cor:nck0}
    Let $(a,b,c)$ be a game state for an $F_4$ Black Hole Zeckendorf game. For all $\a,\g \in \Znn$, such that $\alpha \geq \gamma$, $(3\alpha,1,4\gamma+1)$ is an $N$ position. For all  $\a,\g \in \Znn$ such that $\alpha \geq \gamma+1$, $(3\alpha,1,4\gamma)$ is an $N$ position.
\end{corollary}
\begin{proof}
$(3\alpha,0,4\gamma)$ can be placed immediately after $(3\alpha,1,4\gamma+1)$ and since $(3\alpha,0,4\gamma)$ is a $P$ position for all $\a \geq \g$ as shown in Lemma \ref{lemma:(3,0,4)}, this implies $(3\alpha,1,4\gamma+1)$ is an $N$ position.
Similarly, $(3\alpha,1,4\gamma)$ is an $N$ position for all $\alpha \geq \gamma+1$. For contradiction's sake, assume it is a $P$ position. Then, consider the game tree below.
\begin{center}
    \begin{tikzpicture}
\node (ST1) at (0,0){$\ST{(3\a,1,4\g)}$};
	\node (NS1i) at (-3,-1.5){$\NS{(3(\a-1)+2,0,4\g+1)}$};
		\node (ST2i) at (-3,-3){$\ST{(3(\a-1),1,4\g+1)}$};
			\node [label={[label distance=1mm]270:R.2.A\tsub{2}}] (NS3i) at (-3,-4.5){$\NS{(3(\a-1),0,4\g)}$};
        \node (ST2ii) at (3,-3){$\ST{(3\a,0,4(\g-1)+3)}$};
	\node (NS1ii) at (3,-1.5){$\NS{(3\a,0,4(\g-1)+3)}$};
            \draw  (ST1) -- node[left=3mm] {A\tsub{1}}(NS1i);
            \draw  (ST1) -- node[right=3mm] {A\tsub{2}}(NS1ii);
            \draw  (NS1i) -- node[right] {M}(ST2i);
            \draw   (NS1i) -- node[right=5mm] {S\tsub{3}}(ST2ii);
            \draw   (ST2i) -- node[right] {A\tsub{2}}(NS3i);
        \end{tikzpicture}
\end{center}\pvs
By assumption, $(3\alpha,0,4(\gamma-1)+3)$ is an $N$ position, so if Red places $(3(\alpha-1)+2,0,4\gamma+1)$, Blue must place $(3(\alpha-1),1,4\gamma+1)$. Since $\alpha \geq \gamma+1$, Red wins at R.2.A\tsub{2} by Lemma \ref{lemma:(3,0,4)}. Therefore, it must also be true that Red has a winning strategy for $(3\alpha,1,4\gamma)$ for all $\a \geq \g +1$. Then, $(3\alpha,1,4\gamma)$ is an $N$ position for all $\a \geq \g +1$.
\end{proof}\pvs
\begin{corollary}\label{cor: nca>c} Let $(a,b,c)$ be a game state for an $F_4$ Black Hole Zeckendorf game, with $\a,\g, k_1 \in \Znn$ such that $0 \leq k_1 \leq 2$. Then, the following positions are $P$ positions
\begin{enumerate}
    \item $(3 \alpha+ k_1, 0, 4\gamma)$ such that $\alpha\geq \gamma+1$ and \label{cor:3k0}
    \item $(3 \alpha+ k_1, 0, 4\gamma+1)$ such that $\alpha \geq \gamma$.\label{cor:3k1}
\end{enumerate}
The following positions are $N$ positions
\begin{enumerate}[resume]
    \item $(3 \alpha+ k_1, 0, 4\gamma+2)$ such that $\alpha \geq \gamma+1$ and\label{cor:3k2}
    \item $(3 \alpha+ k_1, 0, 4\gamma+3)$ such that $\alpha \geq \gamma$.\label{cor:3k3}
\end{enumerate}
\end{corollary}
\begin{proof}
    We prove Part \ref{cor:3k0} by inducting on $\g$, letting $\a \geq \g+1$ be arbitrary. As a base case, note that for $\g=0$, $(3\a+k_1,0,0)$ is a $P$ position by Theorem \ref{thrm: allina}. As our inductive hypothesis, assume $(3 \alpha+ k_1, 0, 4\gamma)$ is a $P$ position for all $\alpha\geq\gamma+1$ and consider the tree below, on $(3 \alpha+ k_1, 0, 4(\gamma+1))$ letting $\a \geq \g+2$.\pvs
    \begin{center}
    \begin{tikzpicture}
\node (ST1) at (0,0){$\ST{(3\a+k_1,0,4(\g+1))}$};
	\node (NS1i) at (3,-1.5){$\NS{(3\a+k_1+1,0,4\g+2)}$};
		\node [label={[label distance=1mm]270:B.2.S\tsub{3}}](ST2i) at (3,-3){$\ST{(3\a+k_1+2,0,4\g)}$};
	\node [label={[label distance=1mm]270:R.1.M}](NS1ii) at (-3,-1.5){$\NS{(3(\a-1)+k_1+1,1,4(\g+1))}$};
            \draw  (ST1) -- node[right=3mm] {S\tsub{3}}(NS1i);
            \draw  (ST1) -- node[left=3mm] {M}(NS1ii);
            \draw  (NS1i) -- node[right] {S\tsub{3}}(ST2i);
        \end{tikzpicture}
\end{center}\pvs
By inductive hypothesis B.2.S\tsub{3} is a $P$ position, given that $\a \geq \g +2$ implies $\a \geq \g+1$. If $k_1=0,1$ then R.1.M is an $N$ position by Lemma \ref{lemma:nck12}. If $k_1=2$, then R.1.M is $(3\a,1,4(\g+1))$ which is an $N$ position by Corollary \ref{cor:nck0} since $\alpha \geq \gamma +2$. \pvs
Similarly, we prove Part \ref{cor:3k1} by inducting on $\g$, letting $\a \geq \g$ be arbitrary. As a base case, for $\g=0$, $(3\a+k_1,0,1)$ is a $P$ position by Corollary \ref{coralla}. Then, assume $(3 \alpha+ k_1, 0, 4\gamma+1)$ is a $P$ position for all $\a\geq \g$ and consider the game tree on $(3 \alpha+ k_1, 0, 4(\gamma+1)+1)$, with $\a\geq \g+1$ .
    \begin{center}
    \begin{tikzpicture}
\node (ST1) at (0,0){$\ST{(3\a+k_1,0,4(\g+1)+1)}$};
	\node (NS1i) at (3,-1.5){$\NS{(3\a+k_1+1,0,4\g+3)}$};
		\node [label={[label distance=1mm]270:B.2.S\tsub{3}}](ST2i) at (3,-3){$\ST{(3\a+k_1+2,0,4\g+1)}$};
	\node [label={[label distance=1mm]270:R.1.M}](NS1ii) at (-3,-1.5){$\NS{(3(\a-1)+k_1+1,1,4(\g+1)+1)}$};
            \draw  (ST1) -- node[right=3mm] {S\tsub{3}}(NS1i);
            \draw  (ST1) -- node[left=3mm] {M}(NS1ii);
            \draw  (NS1i) -- node[right] {S\tsub{3}}(ST2i);
        \end{tikzpicture}
\end{center}\pvs
By inductive hypothesis, B.2.S\tsub{3} is a $P$ position given that $\a \geq \g+1$ implies $\a \geq \g$. Again, if $k_1=0,1$ then R.1.M is an $N$ position by Lemma \ref{lemma:nck12}. If $k_1=2$, then R.1.M is $(3\a,1,4(\gamma+1)+1))$ which is an $N$ position by Corollary \ref{cor:nck0} given that $\alpha \geq \gamma+1$.\pvs
As a direct result, $(3 \alpha+ k_1, 0, 4\gamma+2)$ is an $N$ position for all $\alpha \geq \gamma+1$ as claimed in Part \ref{cor:3k2} since the player to move next can place $(3 \alpha+ k_1+1, 0, 4\gamma)$ which is a $P$ position as proven above.\pvs
Likewise, $(3 \alpha+ k_1, 0, 4\gamma+3)$ is an $N$ position for all $(3 \alpha+ k_1, 0, 4\gamma+3)$ when $\alpha \geq \gamma$ as claimed in Part \ref{cor:3k3}, since the player to move next can place $(3 \alpha+ k_1+1, 0, 4\gamma+1)$ which is a $P$ position as proven above.
\end{proof}\pvs
\begin{corollary}\label{cor: nca>c+2}
    Let $(a,b,c)$ be a game state for an $F_4$ Black Hole Zeckendorf game. For all $\a \geq \g$, $(3\a, 1, 4\g+2)$ is an $N$ position.  For all $\a \geq \g +3$, $(3\a, 1, 4\g+3)$ is an $N$ position.
\end{corollary}
\begin{proof}
   The first player to move from $(3\a, 1, 4\g+2)$ can add the second two columns to place 
   $(3\a,0,4\g + 1)$ which is a $P$ position by Corollary \ref{cor: nca>c} Part \ref{cor:3k0} for all $\a \geq \g$. \pvs
   The first player to move from $(3\a, 1, 4\g+3)$ can add the first two columns to place $(3(\a-1)+2,0,4(\g + 1)),$ which is a $P$ position by Corollary \ref{cor: nca>c} Part \ref{cor:3k0} for all $\a -1 \geq \g +2$ which is equivalent to $\a \geq \g+3$.
\end{proof}\pvs
\begin{lemma}\label{lemma:stk0}
    Let $(a,b,c)$ be a game state for an $F_4$ Black Hole Zeckendorf game, with $\a,\g, k_3 \in \Znn$ such that $0 \leq k_3 \leq 3$. $(3\alpha,1,4\gamma+k_3)$ is a $P$ position for all $\alpha \leq  \gamma - 2$.
\end{lemma}

\begin{proof}
    We proceed by strong induction on $\g$, letting $\g \geq 2$. We then assume as our hypothesis that there exists a fixed $\g$ such that $(3\alpha,1,4x+k_3)$ is a $P$ position for all $\alpha \leq  x - 2$ and all $\g \geq x$. Then, consider the game on $(3\a,1,4(\gamma+1)+k_3)$ letting $\alpha \leq \gamma - 1$.\pvs
    \begin{center}
     \begin{tikzpicture}
\node (ST1) at (0,0){$\ST{(3\a,1,4(\g+1)+k_3)}$};
	\node (NS1i) at (-6,-1.5){$\NS{(3(\a-1)+1,2,4(\g+1)+k_3)}$};
		\node [label={[label distance=1mm]270:B.2.MA\tsub{1}}](ST2i) at (-4,-3){$\ST{(3(\a-1),1,4(\g+1)+1+k_3)}$};
    \node (NS1ii) at (-.25,-1.5){$\NS{(3(\a-1)+2,0,4(\g+1)+1+k_3)}$};
	\node (NS1iii) at (5,-1.5){$\NS{(3\a,0,4\g+3+k_3)}$};
        \node (ST2ii) at (5,-3){$\ST{(3\a+1,0,4\g+1+k_3)}$};
            \node (NS2i) at (5,-4.5){$\NS{(3\a+2,0,4(\g-1)+3+k_3)}$};
                \node [label={[label distance=1mm]270:B.3.S\tsub{3}}](ST3i) at (5,-6){$\ST{(3(\a+1),0,4(\g-1)-3+k_3)}$};
                \node [label={[label distance=1mm]270:B.3.M}](ST3ii) at (-2,-6){$\ST{(3\a,1,4(\g-1)+3+k_3)}$};
            \draw  (ST1) -- node[left=3mm] {M}(NS1i);
            \draw  (ST1) -- node[right] {A\tsub{1}}(NS1ii);
            \draw  (ST1) -- node[right=3mm] {A\tsub{2}}(NS1iii);
            \draw  (NS1i) -- node[right] {A\tsub{1}}(ST2i);
            \draw  (NS1ii) -- node[left=3mm] {M}(ST2i);
            \draw  (NS1iii) -- node[right] {S\tsub{3}}(ST2ii);
            \draw  (ST2ii) -- node[right] {S\tsub{3}}(NS2i);
            \draw  (NS2i) -- node[right] {S\tsub{3}}(ST3i);
            \draw  (NS2i) -- node[left=5mm] {M}(ST3ii);
        \end{tikzpicture}
    \end{center}\pvs
While it is not immediately clear that B.2.MA\tsub{2} is a $P$ position, note that the possible moves from here are the same as the possibilities from $(3\a,1,4(\g+1))$ and that the necessary inequalities still hold, given that $\a-1<\a$. If Red always chooses to merge or add the first two columns, the game reduces to Blue placing $(0,1,4(\g+1)+k_3+1)$, which is a $P$ position for all $4(\g+1)+k_3+1 \neq 0,1,6$ by Corollary \ref{corallc}. Since we require $\g-1 \geq \a \geq 0 \implies \g \geq 1$, we do not need to consider these exceptions. Showing that one of B.3.M and B.3.S\tsub{3} is a $P$ position takes a bit more effort, as which position wins depends on the values of $\a,\g$ and $k_3$.\pvs

First, suppose $k_3 = 0$. Then, $\text{B.3.M} = (3\a,1,4(\gamma-1)+3)$ and $\text{B.3.S\tsub{3}} = (3(\alpha+1),0,4(\gamma-2)+1)$. B.3.M is a $P$ position by inductive hypothesis for all $\alpha \leq (\gamma-1)-2$, and we assumed $\alpha \leq \gamma-1$. Therefore, we only need to show that $(3\a,1,4(\g+1))$ is also a $P$ position when $\g=\a+1$ and $\g = \a+2$. Note that when $\gamma = \alpha +1$, $\text{B.3.S\tsub{3}} = (3(\alpha+1),0,4(\alpha-1)+1)$ which is a $P$ position by Corollary \ref{cor: nca>c} Part \ref{cor:3k1}. Similarly, when  $\gamma = \alpha + 2$, $\text{B.3.S\tsub{3}} = (3(\alpha+1),0,4\alpha+1)$, which is a $P$ position by Corollary \ref{cor: nca>c} Part \ref{cor:3k1}.\pvs
Next, suppose $k_3 = 1$, implying that $\text{B.3.M} = (3\alpha,1,4\g)$. $\text{B.3.M}$ is a $P$ position by inductive hypothesis for all $\alpha \leq \g-2$,and since we let $\a \leq \g-1$ so we only need to consider when $\g=\a+1$. We consider this in the game tree below. 
         \begin{center}
    \begin{tikzpicture}
\node (ST1) at (0,0){$\ST{(3\a,1,4(\a+1))}$};
	\node (NS1i) at (-5,-1.5){$\NS{(3(\a-1)+1, 2, 4(\a+1))}$};
        \node [label={[label distance=1mm]270:B.2.MA\tsub{1}}](ST2i) at (-2,-3){$\ST{(3(\a-1), 1,4(\a+1)+1)}$};
	\node (NS1ii) at (0,-1.5){$\NS{(3(\a-1)+2, 0,4(\a+1)+1)}$};
    \node (NS1iii) at (5,-1.5){$\NS{(3\a, 0,4\a+3)}$};
        \node [label={[label distance=1mm]270:B.2.S\tsub{3}}](ST2ii) at (5,-3){$\ST{(3\a+1, 0,4\a+1)}$};

            \draw  (ST1) -- node[left=6mm] {M}(NS1i);
            \draw  (ST1) -- node[right] {A\tsub{1}}(NS1ii);
            \draw  (ST1) -- node[right=6mm] {A\tsub{2}}(NS1iii);
            \draw  (NS1i) -- node[left=3mm] {A\tsub{1}}(ST2i);
            \draw  (NS1ii) -- node[right=3mm] {M}(ST2i);
            \draw  (NS1iii) -- node[right] {S\tsub{3}}(ST2ii);
        \end{tikzpicture}
\end{center}\pvs
B.2.MA\tsub{2} is a $P$ position by the inductive hypothesis and B.2.S\tsub{3} is a $P$ position by Corollary \ref{cor: nca>c} Part \ref{cor:3k1}.\pvs
Then, suppose $k_3=2$, implying that $\text{B.3.M} = (3\a,1,4\g+1)$. This is a $P$ position by inductive hypothesis for all $\a \leq \g -2$, and we assume $\a \leq \g -1$, so again, we only need to consider when $\g=\a+1$ as we do in the game tree below. \pvs
  \begin{center}
    \begin{tikzpicture}
\node [label={[label distance=1mm]90:B.3.M}](ST1) at (0,0){$\ST{(3\a, 1, 4(\alpha +1)+1)}$};
	\node (NS1i) at (-5,-1.5){$\NS{(3(\a-1)+1, 2, 4(\a +1)+1)}$};
		\node [label={[label distance=1mm]270:B.4.MA\tsub{1}}](ST2i) at (-5,-3){$\ST{(3(\a-1), 1, 4(\a+1)+2)}$};
    \node (NS1ii) at (.5,-1.5){$\NS{(3(\a-1)+2, 0, 4(\a +1)+2)}$};
	\node (NS1iii) at (5,-1.5){$\NS{(3\a, 0, 4(\a+1))}$};
        \node (ST2ii) at (5,-3){$\ST{(3\a+1, 0, 4\a+2)}$};
            \node (NS2i) at (5,-4.5){$\NS{(3\a+2, 0, 4\a)}$};
                \node [label={[label distance=1mm]270:B.5.M}] (ST3i) at (5,-6){$\ST{(3\a, 1, 4\a)}$};
            \draw  (ST1) -- node[left=5mm] {M}(NS1i);
            \draw  (ST1) -- node[left] {A\tsub{1}}(NS1ii);
            \draw  (ST1) -- node[right=5mm] {A\tsub{2}}(NS1iii);
            \draw  (NS1i) -- node[right] {A\tsub{1}}(ST2i);
            \draw  (NS1ii) -- node[right=5mm] {M}(ST2i);
            \draw  (NS1iii) -- node[right] {S\tsub{3}}(ST2ii);
            \draw  (ST2ii) -- node[right] {S\tsub{3}}(NS2i);
            \draw  (NS2i) -- node[right] {M}(ST3i);
        \end{tikzpicture}
\end{center}\pvs
B.4.MA\tsub{2} is a $P$ position by the inductive hypothesis. We continue the tree from B.5.M in Figure \ref{fig:(3a,1,4a)}.
  \begin{figure}
\begin{center}
    \begin{tikzpicture}
\node [label={[label distance=1mm]90:B.5.M}](ST1) at (0,0){$\ST{(3\a, 1, 4\a)}$};
	\node (NS1i) at (-5.5,-1.5){$\NS{(3(\a-1)+1, 2, 4\a)}$};
		\node (ST2i) at (-2,-3){$\ST{(3(\a-1), 1, 4\a+1)}$};
            \node (NS2i) at (-5.5,-4.5){$\NS{(3(\a-2)+1, 2, 4\a+1)}$};
                \node [label={[label distance=1mm]270:B.7.MA\tsub{1}}] (ST3i) at (-5.5,-6){$\ST{(3(\a-2), 1, 4\a+2)}$};
            \node (NS2ii) at (-.25,-4.5){$\NS{(3(\a-2)+2, 0, 4\a+2)}$};
            \node (NS2iii) at (4,-4.5){$\NS{(3(\a-1), 0, 4\a)}$};
                \node (ST3ii) at (0,-6){$\ST{(3(\a-1)+1, 0, 4(\a-1)+2)}$};
                \node (NS3i) at (0,-7.5){$\NS{(3(\a-1)+2, 0, 4(\a-1))}$};
                \node (ST4i) at (0,-9){$\ST{(3(\a-1), 1, 4(\a-1))}$};
                \node (ST5i) at (0,-10.5){$\ST{. . .}$};
                \node [label={[label distance=1mm]270:B.9.M}](ST6i) at (0,-12){$\ST{(3(\a-r), 1, 4(\a-r))}$};
    \node (NS1ii) at (0,-1.5){$\NS{(3(\a-1)+2, 0, 4\a+1)}$};
	\node (NS1iii) at (5.5,-1.5){$\NS{(3\a, 0, 4(\a-1)+3)}$};
        \node (ST2ii) at (5.5,-3){$\ST{(3\a+1, 0, 4(\a-1)+1)}$};
            \node (NS2iv) at (5.5,-7.5){$\NS{(3\a+2, 0, 4(\a-2)+3)}$};
                \node [label={[label distance=1mm]270:B.7.S\tsub{3}}](ST3iii) at (5.5,-9){$\ST{(3(\a+1), 0, 4(\a-2)+1)}$};
            \draw  (ST1) -- node[left=6mm] {M}(NS1i);
            \draw  (ST1) -- node[right] {A\tsub{1}}(NS1ii);
            \draw  (ST1) -- node[right=6mm] {A\tsub{2}}(NS1iii);
            \draw  (NS1i) -- node[right =4mm] {A\tsub{1}}(ST2i);
            \draw  (NS1ii) -- node[right=3mm] {M}(ST2i);
            \draw  (NS1iii) -- node[right] {S\tsub{3}}(ST2ii);
            \draw  (ST2i) -- node[left=4mm] {M}(NS2i);
            \draw  (ST2i) -- node[right] {A\tsub{1}}(NS2ii);
            \draw  (ST2i) -- node[right=4mm] {A\tsub{2}}(NS2iii);
            \draw  (ST2ii) -- node[right] {S\tsub{3}}(NS2iv);
            \draw  (NS2i) -- node[right] {A\tsub{1}}(ST3i);
            \draw  (NS2ii) -- node[left=5mm] {M}(ST3i);
            \draw  (NS2iii) -- node[right=4mm] {S\tsub{3}}(ST3ii);
            \draw  (NS2iv) -- node[right] {S\tsub{3}}(ST3iii);
            \draw  (ST3ii) -- node[right] {S\tsub{3}}(NS3i);
            \draw  (NS3i) -- node[right] {M}(ST4i);
            \draw  (ST4i) -- (ST5i);
            \draw  (ST5i) -- (ST6i);
        \end{tikzpicture}
\end{center}
 \caption{Game Tree for the Game State $(3\a,1,4\a)$}
\label{fig:(3a,1,4a)}
  \end{figure}\pvs
In Figure \ref{fig:(3a,1,4a)}, B.7.MA\tsub{1} is a $P$ position by inductive hypothesis and B.7.S\tsub{3} is a $P$ position by Corollary \ref{cor: nca>c} Part \ref{cor:3k1}. Otherwise, the game reduces to $(3(\a-r),1,4(\a-r))$ at B.9.M, so the game tree repeats for $r$ rounds, until Blue places a state which we have shown to win by our inductive hypothesis, wins by Corollary \ref{cor: nca>c} Part \ref{cor:3k1}, or places $(0,1,0)$ which trivially wins. This concludes our proof of the statement that $(3\a,1,4\g+2)$ is a $P$ position for all $\a \leq \g-2$.\pvs
Lastly, we consider when $k_3= 3$ so that $\text{B.3.M} = (3\a,1,4\g+2))$. By the inductive hypothesis, this is a $P$ position for all $\a \leq \g-2$, and since we let $\a \leq \g-1$, we again must only consider $\g = \a +1$. Here, $\text{B.3.S\tsub{3}} = (3(\a+1),0,4\a)$ which is a $P$ position by Lemma \ref{lemma:(3,0,4)}.\pvs
Thus, Player $2$ has a winning strategy for $(3\a,1,4\g+k_3)$ for all $\a \leq \g -2$.
\end{proof}\pvs
\begin{corollary}\label{cor:nc010}
    Let $(a,b,c)$ be a game state for an $F_4$ Black Hole Zeckendorf game, with $\a,\g \in \Znn$. For all $\a=\g$, $(3\a,1,4\g)$ is a $P$ position.
\end{corollary}
\begin{proof}
We showed in Lemma \ref{lemma:stk0} that $(3\a,1,4\a)$ is a $P$ position under the assumption that $(3\a,1,4\g+k_3)$ is a $P$ position for all $\a \leq \g -2$. Since we proved that Lemma \ref{lemma:stk0} is in fact true, then it follows that $(3\a,1,4\g)$ is a $P$ position for all $\a = \g$.
\end{proof}\pvs
We now know that $(3\a, 1, 4 \g +k_3)$ is a $P$ position for all $\a \leq \g -2$, and from Corollary \ref{cor:nck0} and Corollary \ref{cor: nca>c+2}, we know that $(3\a,1,4\g + k_3)$ is a $P$ position for all $\a \geq \g +3$. By considering each $k_3$ individually, we determine the which positions are $P$ and $N$ for $(3 \a, 1, 4 \g + k_3)$ for any $\a, \g, k_3 \in \Znn$ such that $0 \leq k_3 \leq 3$. \pvs
\begin{thm} \label{thm:wink1=0} Let $(a,b,c)$ be a game state for an $F_4$ Black Hole Zeckendorf Game with $\a,\g \in \Znn$.
\begin{enumerate}
    \item $(3\a,1,4\g)$ is a $P$ position for all $\a\leq \g$ and $(3\a,1,4\g)$ is an $N$ position for all $\a\geq \g +1$. \label{wink2=0}
    \item $(3\a,1,4\g+1)$ is a $P$ position for all $\a\leq \g-1$ and $(3\a,1,4\g+1)$ is an $N$ position for all $\a\geq \g$. \label{wink2=1}
    \item $(3\a,1,4\g+2)$ is a $P$ position for all $\a\leq \g-2$ and $(3\a,1,4\g+2)$ is an $N$ position for all $\a\geq \g-1$.\label{wink2=2}
    \item $(3\a,1,4\g+3)$ is a $P$ position for all $\a\leq \g+2$ and $(3\a,1,4\g+3)$ is an $N$ position for all $\a\geq \g+3$.\label{wink2=3}
\end{enumerate}
    \begin{proof} We proceed by proving each part of the theorem individually.\\
\vspace{.25cm}\\
        \textbf{Proof when }\bm{$k_3=0$}.\par
        First, we consider when $k_3 = 0$. From Corollary \ref{cor:nck0}, we know that $(3\a,1,4\g)$ is an $N$ position for all $\a\geq \g +1$, and from Lemma $\ref{lemma:stk0}$, we know that $(3\a,1,4\g)$ is a $P$ position for all $\a \leq \g -2$. Moreover, we showed that $(3\a,1,4\a)$ is a $P$ position in Corollary \ref{cor:nc010}. Therefore, it is only necessary to show that $(3\a,1,4\g)$ is a $P$ position for $\g = \a +1$, which we show in the game tree below.
        
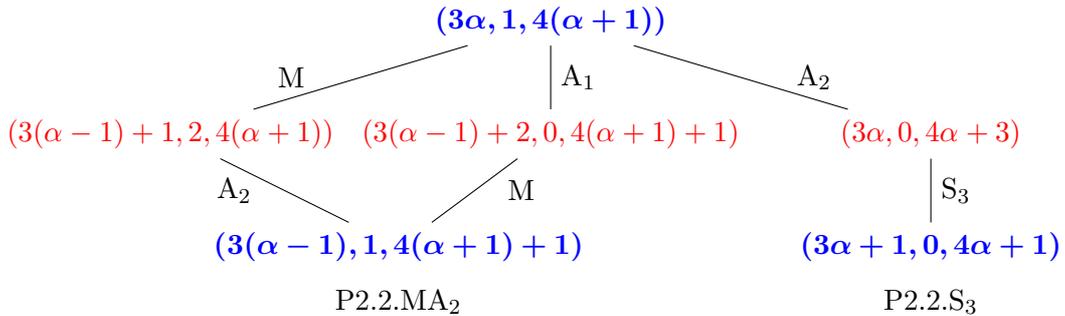
\begin{center}
    \begin{tikzpicture}
\node (ST1) at (0,0){$\ST{(3\a,1,4(\a+1))}$};
	\node (NS1i) at (-5,-1.5){$\NS{(3(\a-1)+1, 2, 4(\a+1))}$};
        \node [label={[label distance=1mm]270:B.2.MA\tsub{1}}](ST2i) at (-2,-3){$\ST{(3(\a-1), 1,4(\a+1)+1)}$};
	\node (NS1ii) at (0,-1.5){$\NS{(3(\a-1)+2, 0,4(\a+1)+1)}$};
    \node (NS1iii) at (5,-1.5){$\NS{(3\a, 0,4\a+3)}$};
        \node [label={[label distance=1mm]270:B.2.S\tsub{3}}](ST2ii) at (5,-3){$\ST{(3\a+1, 0,4\a+1)}$};

            \draw  (ST1) -- node[left=6mm] {M}(NS1i);
            \draw  (ST1) -- node[right] {A\tsub{1}}(NS1ii);
            \draw  (ST1) -- node[right=6mm] {A\tsub{2}}(NS1iii);
            \draw  (NS1i) -- node[left=3mm] {A\tsub{1}}(ST2i);
            \draw  (NS1ii) -- node[right=3mm] {M}(ST2i);
            \draw  (NS1iii) -- node[right] {S\tsub{3}}(ST2ii);
        \end{tikzpicture}
\end{center}\pvs
B.2.MA\tsub{2} is a $P$ position by Lemma \ref{lemma:stk0} and B.2.S\tsub{3} is a $P$ position by Corollary \ref{cor: nca>c} Part \ref{cor:3k1}. Thus, $(3\a,1,4\g)$ is a $P$ position for all $\a\leq \g$ and $(3\a,1,4\g)$ is an $N$ position for all $\a\geq \g +1$.\\
\vspace{.25cm}\\
        \textbf{Proof when }\bm{$k_3=1$}.\par
 Next, we consider when $k_3= 1$. From Corollary \ref{cor:nck0}, we know that $(3\a,1,4\g+1)$ is an $N$ position for all $\a \geq \g$, and from Lemma $\ref{lemma:stk0}$, $(3\a,1,4\g+1)$ is a $P$ position for all $\a \leq \g -2$. Thus, we only need to consider $\g = \a +1$, which we do in the game tree below.
   \begin{center}
    \begin{tikzpicture}
\node (ST1) at (0,0){$\ST{(3\a,1,4(\a+1)+1)}$};
	\node (NS1i) at (-5,-1.5){$\NS{(3(\a-1)+1, 2, 4(\a+1)+1)}$};
        \node [label={[label distance=1mm]270:B.2.MA\tsub{1}}](ST2i) at (-2,-3){$\ST{(3(\a-1), 1,4(\a+1)+2)}$};
	\node (NS1ii) at (.5,-1.5){$\NS{(3(\a-1)+2, 0,4(\a+1)+2)}$};
    \node (NS1iii) at (5,-1.5){$\NS{(3\a, 0,4(\a+1))}$};
        \node (ST2ii) at (5,-3){$\ST{(3\a+1, 0,4\a+2)}$};
            \node (NS2i) at (5,-4.5){$\NS{(3\a+2, 0,4\a)}$};
                \node [label={[label distance=1mm]270:B.3.M}](ST3i) at (5,-6){$\ST{(3\a, 1,4\a)}$};
            \draw  (ST1) -- node[left=5mm] {M}(NS1i);
            \draw  (ST1) -- node[right] {A\tsub{1}}(NS1ii);
            \draw  (ST1) -- node[right=3mm] {A\tsub{2}}(NS1iii);
            \draw  (NS1i) -- node[left=2mm] {A\tsub{1}}(ST2i);
            \draw  (NS1ii) -- node[right=2mm] {M}(ST2i);
            \draw  (NS1iii) -- node[right] {S\tsub{3}}(ST2ii);
            \draw  (ST2ii) -- node[right] {S\tsub{3}}(NS2i);
            \draw  (NS2i) -- node[right] {M}(ST3i);
        \end{tikzpicture}
\end{center}\pvs
B.2.MA\tsub{1} is a $P$ position by Lemma \ref{lemma:stk0} and B.3.M is a $P$ position by Corollary \ref{cor:nc010}. Hence, $(3\a,1,4\g+1)$ is a $P$ position for all $\a\leq \g-1$ and $(3\a,1,4\g+1)$ is an $N$ position for all $\a\geq \g$.\\
\vspace{.25cm}\\
        \textbf{Proof when }\bm{$k_3=2$}.\par
Then, we consider $k_3=2$. From Corollary \ref{cor: nca>c+2}, we know that $(3\a,1,4\g+2)$ is an $N$ position for all $\a \geq \g$, and from Lemma $\ref{lemma:stk0}$, $(3\a,1,4\g+2)$ is a $P$ position for all $\a \leq \g -2$, so again, we only need to consider when $\g = \a +1$, as we do in the game play tree below.
   \begin{center}
    \begin{tikzpicture}
\node (ST1) at (0,0){$\ST{(3\a,1,4(\a+1)+2)}$};
    \node (NS1i) at (0,-1.5){$\NS{(3\a, 0,4(\a+1)+1)}$};
        \node (ST2i) at (0,-3){$\ST{(3\a+1, 0,4\a+3)}$};
            \node [label={[label distance=1mm]270:R.2.S\tsub{3}}] (NS2i) at (0,-4.5){$\NS{(3\a+2, 0,4\a+1)}$};
            \draw  (ST1) -- node[right] {A\tsub{2}}(NS1i);
            \draw  (NS1i) -- node[right] {S\tsub{3}}(ST2i);
            \draw  (ST2i) -- node[right] {S\tsub{3}}(NS2i);
        \end{tikzpicture}
\end{center}\pvs
R.2.S\tsub{3} is an $N$ position by Corollary \ref{cor: nca>c} Part \ref{cor:3k1}, so $(3\a,1,4\g+2)$ is a $P$ position for all $\a \leq \g -2$ and $(3\a,1,4\g+2)$ is an $N$ position for all $\a \geq \g -1$.\\
\vspace{.25cm}\\
\textbf{Proof when }\bm{$k_3=3$}.\par
Finally, we consider when $k_3 = 3$. By Lemma \ref{lemma:stk0}, $(3\a,1,4\g+3)$ is a $P$ position for all $\a \leq \g -2$ and by Corollary \ref{cor: nca>c+2}, $(3\a,1,4\g+3)$ is an $N$ position for all $\a \geq \g +3$. Thus, we must consider when $\a-2 \leq \g \leq \a + 1$. We begin with the game tree such that $\g= \a-2$.
\begin{center}
    \begin{tikzpicture}
\node (ST1) at (0,0){$\ST{(3\a,1,4(\a-2)+3)}$};
	\node (NS1i) at (-5,-1.5){$\NS{(3(\a-1)+1, 2, 4(\a-2)+3)}$};
        \node [label={[label distance=1mm]270:B.2.MA\tsub{1}}](ST2i) at (-2,-3){$\ST{(3(\a-1), 1,4(\a-1))}$};
	\node (NS1ii) at (.5,-1.5){$\NS{(3(\a-1)+2, 0,4(\a-1))}$};
    \node (NS1iii) at (5,-1.5){$\NS{(3\a, 0,4(\a-2)+2)}$};
        \node [label={[label distance=1mm]270:B.2.S\tsub{3}}](ST2ii) at (5,-3){$\ST{(3\a+1, 0,4(\a-2))}$};

            \draw  (ST1) -- node[left=5mm] {M}(NS1i);
            \draw  (ST1) -- node[right] {A\tsub{1}}(NS1ii);
            \draw  (ST1) -- node[right=5mm] {A\tsub{2}}(NS1iii);
            \draw  (NS1i) -- node[left=2mm] {A\tsub{1}}(ST2i);
            \draw  (NS1ii) -- node[right=2mm] {M}(ST2i);
            \draw  (NS1iii) -- node[right] {S\tsub{3}}(ST2ii);
        \end{tikzpicture}
\end{center}\pvs
B.2.MA\tsub{1} is a $P$ position by Corollary \ref{cor:nc010}. Additionally, B.2.S\tsub{3} is a $P$ position by Corollary \ref{cor: nca>c} Part \ref{cor:3k0}. Next, we consider the game tree such that $\g = \a - 1$.
\begin{center}
    \begin{tikzpicture}
\node (ST1) at (0,0){$\ST{(3\a,1,4(\a-1)+3)}$};
	\node (NS1i) at (-5,-1.5){$\NS{(3(\a-1)+1, 2, 4(\a-1)+3)}$};
        \node [label={[label distance=1mm]270:B.2.MA\tsub{1}}](ST2i) at (-2,-3){$\ST{(3(\a-1), 1,4\a)}$};
	\node (NS1ii) at (.5,-1.5){$\NS{(3(\a-1)+2, 0,4\a)}$};
    \node (NS1iii) at (5,-1.5){$\NS{(3\a, 0,4(\a-1)+2)}$};
        \node [label={[label distance=1mm]270:B.2.S\tsub{3}}](ST2ii) at (5,-3){$\ST{(3\a+1, 0,4(\a-1))}$};

            \draw  (ST1) -- node[left=5mm] {M}(NS1i);
            \draw  (ST1) -- node[right] {A\tsub{1}}(NS1ii);
            \draw  (ST1) -- node[right=5mm] {A\tsub{2}}(NS1iii);
            \draw  (NS1i) -- node[left=2mm] {A\tsub{1}}(ST2i);
            \draw  (NS1ii) -- node[right=2mm] {M}(ST2i);
            \draw  (NS1iii) -- node[right] {S\tsub{3}}(ST2ii);
        \end{tikzpicture}
\end{center}\pvs
B.2.MA\tsub{1} is a $P$ position by the proof of Part \ref{wink2=0} of this Theorem. Additionally, B.2.S\tsub{3} is a $P$ position by Corollary \ref{cor: nca>c} Part \ref{cor:3k0}. Then, we continue to $\g=\a$.
\begin{center}
    \begin{tikzpicture}
\node (ST1) at (0,0){$\ST{(3\a,1,4\a+3)}$};
	\node (NS1i) at (-5,-1.5){$\NS{(3(\a-1)+1, 2, 4\a+3)}$};
        \node [label={[label distance=1mm]270:B.2.MA\tsub{1}}](ST2i) at (-2,-3){$\ST{(3(\a-1), 1,4(\a+1))}$};
	\node (NS1ii) at (.5,-1.5){$\NS{(3(\a-1)+2, 0,4(\a+1))}$};
    \node (NS1iii) at (5,-1.5){$\NS{(3\a, 0,4\a+2)}$};
        \node (ST2ii) at (5,-3){$\ST{(3\a+1, 0,4\a)}$};
            \node(NS2i) at (5,-4.5){$\NS{(3\a+2, 0,4(\a-1)+2)}$};
                \node [label={[label distance=1mm]270:B.3.S\tsub{3}}](ST3i)at (5,-6){$\ST{(3(\a+1), 0,4(\a-1))}$};
            \draw  (ST1) -- node[left=5mm] {M}(NS1i);
            \draw  (ST1) -- node[right] {A\tsub{1}}(NS1ii);
            \draw  (ST1) -- node[right=5mm] {A\tsub{2}}(NS1iii);
            \draw  (NS1i) -- node[left=2mm] {A\tsub{1}}(ST2i);
            \draw  (NS1ii) -- node[right=2mm] {M}(ST2i);
            \draw  (NS1iii) -- node[right] {S\tsub{3}}(ST2ii);
            \draw  (ST2ii) -- node[right] {S\tsub{3}}(NS2i);
            \draw  (NS2i) -- node[right] {S\tsub{3}}(ST3i);
        \end{tikzpicture}
\end{center}\pvs
When $\a=\g=0$, the game starts at $(0,1,3)$ which is a $P$ position by Corollary \ref{corallc}. For $\a \in \Zp$ B.2.MA\tsub{1} is a $P$ position by Lemma \ref{lemma:stk0} and B.3.S\tsub{3} is a $P$ position by Corollary \ref{cor: nca>c} Part \ref{cor:3k0}. Then, we conclude with the case where $\g= \a +1$, which is essentially identical to the case for $\g = \a$.
\begin{center}
    \begin{tikzpicture}
\node (ST1) at (0,0){$\ST{(3\a,1,4(\a+1)+3)}$};
	\node (NS1i) at (-5,-1.5){$\NS{(3(\a-1)+1, 2, 4(\a+1)+3)}$};
        \node [label={[label distance=1mm]270:B.2.MA\tsub{1}}](ST2i) at (-2,-3){$\ST{(3(\a-1), 1,4(\a+2))}$};
	\node (NS1ii) at (.5,-1.5){$\NS{(3(\a-1)+2, 0,4(\a+2))}$};
    \node (NS1iii) at (5,-1.5){$\NS{(3\a, 0,4(\a+1)+2)}$};
        \node (ST2ii) at (5,-3){$\ST{(3\a+1, 0,4(\a+1))}$};
            \node(NS2i) at (5,-4.5){$\NS{(3\a+2, 0,4\a+2)}$};
                \node [label={[label distance=1mm]270:B.3.S\tsub{3}}](ST3i)at (5,-6){$\ST{(3(\a+1), 0,4\a)}$};
            \draw  (ST1) -- node[left=5mm] {M}(NS1i);
            \draw  (ST1) -- node[right] {A\tsub{1}}(NS1ii);
            \draw  (ST1) -- node[right=5mm] {A\tsub{2}}(NS1iii);
            \draw  (NS1i) -- node[left=2mm] {A\tsub{1}}(ST2i);
            \draw  (NS1ii) -- node[right=2mm] {M}(ST2i);
            \draw  (NS1iii) -- node[right] {S\tsub{3}}(ST2ii);
            \draw  (ST2ii) -- node[right] {S\tsub{3}}(NS2i);
            \draw  (NS2i) -- node[right] {S\tsub{3}} (ST3i);
        \end{tikzpicture}
\end{center}\pvs
Again, B.2.MA\tsub{1} is a $P$ position by Lemma \ref{lemma:stk0} and B.3.S\tsub{3} is a $P$ position by Corollary \ref{cor: nca>c} Part \ref{cor:3k0}. \pvs

This concludes our proof of Theorem \ref{thm:wink1=0}
    \end{proof}
\end{thm}
From here, we can begin to prove some of the statements claimed in Figure \ref{fig:winners(a,0,c)}. For a solution to be truly constructive, we want a path from any game position to a $P$ position. In Lemma \ref{lemma:B=0} we prove that certain states are indeed $P$ positions below, and then provide constructive strategies from $P$ positions in Theorem \ref{thm:B=0}.\pvs
\begin{lemma}\label{lemma:B=0}
     Let $(a,0,c)$ be a game state for an the $F_4$ Black Hole Zeckendorf game $\a,\g\in \Znn$. The following positions are $P$ positions
         \begin{enumerate}
             \item $(3\a,0,4\g)$ such that $\a \geq \g$, and \label{lemma:00}
              \item $(3\a+1,0,4\g)$ for all $\a,\g$. \label{lemma:10}
         \end{enumerate}
\end{lemma}
\begin{proof} We prove each statement individually.\pvs
\textbf{Proof of Part 1.}  \par
    We first show that $(3\a,0,4\g)$ is a $P$ position for all $\a \geq \g$, noting that when $\a=\g=0$, $(0,0,0)$ is trivially a $P$ position. By Corollary \ref{cor: nca>c} Part \ref{cor:3k0}, $(3\a,0,4\g)$ is a $P$ position for all $\a \geq \g +1$, so it is only necessary to consider the game when $\g= \a$, as we do below.
      \begin{center}
    \begin{tikzpicture}
\node (ST1) at (0,0){$\ST{(3\a,0,4\a)}$};
	\node (NS1i) at (0,-1.5){$\NS{(3\a+1,0,4(\a-1)+2)}$};
    \node [label={[label distance=1mm]270: B.2.S\tsub{3}}](ST2i) at (0,-3){$\ST{(3\a+2,0,4(\a-1))}$};
            \draw  (ST1) -- node[right] {S\tsub{3}}(NS1i);
            \draw  (NS1i) -- node[right] {S\tsub{3}}(ST2i);
        \end{tikzpicture}
\end{center}\pvs
B.2.S\tsub{3} is a $P$ position by Corollary \ref{cor: nca>c} Part \ref{cor:3k0} so $(3\a,0,4\g)$ is a $P$ position for all $\a \geq \g$.\\
\vspace{.25cm}\\
\textbf{Proof of Part 2.} \par
To conclude this lemma, we show that $(3\a+1,0,4\g)$ is a $P$ position for all $\a,\g$, noting that when $\a=\g=0$, $(1,0,0)$ is trivially a $P$ position. From Corollary \ref{cor: nca>c} Part \ref{cor:3k0}, we already know it is a $P$ position for all  $ \a \geq \g+1$. Then, consider the tree below.
\begin{center}
    \begin{tikzpicture}
\node (ST1) at (0,0){$\ST{(3\a+1,0,4\g)}$};
	\node (NS1i) at (0,-1.5){$\NS{(3\a+2,0,4(\g-1)+2)}$};
    \node [label={[label distance=1mm]270: B.2.S\tsub{3}}](ST2i) at (3,-3){$\ST{(3(\a+1),0,4(\g-1))}$};
    \node [label={[label distance=1mm]270: B.2.M}](ST2ii) at (-3,-3){$\ST{(3\a,1,4(\g-1)+2)}$};
            \draw  (ST1) -- node[right] {S\tsub{3}}(NS1i);
            \draw  (NS1i) -- node[right=5mm] {S\tsub{3}}(ST2i);
            \draw  (NS1i) -- node[left=5mm] {M}(ST2ii);
        \end{tikzpicture}
\end{center}\pvs
B.2.M is a $P$ position for all $\a \leq \g-3$ by Theorem \ref{thm:wink1=0} Part \ref{wink2=2} and B.2.S\tsub{3} is a $P$ position for all $\a \geq \g-2$ as established by the proof of Part \ref{lemma:00}. Therefore, $(3\a+1,0,4\g)$ is a $P$ position for all $\a,\g \in \Znn$.
\end{proof}
Lemma \ref{lemma:B=0} is sufficient information for us to now construct a winning strategy from any $P$ position. Here, we aim to show that if a player has placed a $P$ position, they can reduce the game to another $P$ position.
\begin{thm}\label{thm:B=0} 
Let $(a,0,c)=(3\a+k_1,0,4\g+k_3)$ be a board setup for an $F_4$ Black Hole Zeckendorf game. For $\a,\g,k_1,k_3 \in \Znn$, $0 \leq k_1 \leq 2$, and $0 \leq k_3 \leq 3$, the following positions are $P$ positions.
\begin{enumerate}
    \item $c \equiv 1 \pmod 4$: \label{thm:k2=1}
    \begin{enumerate}
        \item $(3\a,0,4\g+1)$ for all $\a \geq \g-1$. \label{thm:st01}
        \item $(3\a+1,0,4\g+1)$ for all $\a,\g$. \label{thm:st11}
        \item $(3\a+2,0,4\g+1)$ for all $\a \geq \g$. \label{thm:st21}
    \end{enumerate}
    \item $c \equiv 2 \pmod 4$: \label{thm:k2=2}
    \begin{enumerate}
        \item $(3\a+1,0,4\g+2)$ for all $\a \leq \g$. \label{thm:st12}
    \end{enumerate}
    \item $c \equiv 3 \pmod 4$:\label{thm:k2=3}
    \begin{enumerate}
        \item $(3\a+1,0,4\g+3)$ for all $\a \leq \g -1$. \label{thm:st13}
    \end{enumerate}
    \item $c \equiv 0 \pmod 4$:\label{thm:k2=0}
        \begin{enumerate}
        \item $(3\a,0,4\g)$ for all $\a \geq \g$. \label{thm:st00}
        \item $(3\a+1,0,4\g)$ for all $\a,\g$. \label{thm:st10}
        \item $(3\a+2,0,4\g)$ for all $\a \geq \g+1$. \label{thm:st20}
    \end{enumerate}
\end{enumerate}
\end{thm}
\begin{proof}
We work through each section of the theorem, referencing a base case as an example, and then demonstrating the winning strategy. Since we aim to create a constructive solution, we assume that both players are unfamiliar with Lemma \ref{lemma:nck12}, and play out the game when players choose to set the board as $(3\a+1, 1, 4\g+k_3)$ and $(3\a+2, 1, 4\g+k_3)$, although we already know that game state loses. \\
\vspace{.25cm}\\
\textbf{Proof for} $\bm{c \equiv 1\pmod 4}$.\\
\textit{Proof of Part} \ref{thm:st01}:
We first consider the game on $(3\a,0,4\g+1)$ where $\a \geq \g-1$. As base cases, we see that $(0,0,1)$ is trivially a $P$ position and $(0,0,5)$ is a $P$ position by Theorem \ref{thrm: allinc}. Then consider the game tree below.
  \begin{center}
    \begin{tikzpicture}
\node (ST1) at (0,0){$\ST{(3\a,0,4\g+1)}$};
	\node (NS1i) at (3,-1.5){$\NS{(3\a+1,0,4(\g-1)+3)}$};
		\node [label={[label distance=1mm]270: B.2.S\tsub{3}}](ST2i) at (3,-3){$\ST{(3\a+2,0,4(\g-1)+1)}$};
    \node (NS1ii) at (-3,-1.5){$\NS{(3(\a-1)+1,1,4\g+1)}$};
        \node [label={[label distance=1mm]270: B.2.A\tsub{2}}](ST2ii) at (-3,-3){$\ST{(3(\a-1)+1,0,4\g)}$};
            \draw  (ST1) -- node[right=4mm] {S\tsub{3}}(NS1i);
            \draw  (ST1) -- node[left=4mm] {M}(NS1ii);
            \draw  (NS1i) -- node[right] {S\tsub{3}}(ST2i);
            \draw  (NS1ii) -- node[right] {A\tsub{2}}(ST2ii);
        \end{tikzpicture}
\end{center}\pvs
B.2.A\tsub{2} is a $P$ position by Lemma \ref{lemma:B=0} Part \ref{lemma:10} and B.2.S\tsub{3} is a $P$ position by Corollary \ref{cor: nca>c} Part \ref{cor:3k1}, given that $\a \geq \g -1$. Therefore, $(3\a,0,4\g+1)$ is a $P$ position for all $\a \geq \g-1$. \\
\vspace{.25cm}\\
\textit{Proof of Part} \ref{thm:st11}:
Next, we consider the game on $(3\a+1,0,4\g+1)$, noting that as a base case, $(1,0,1)$ is trivially a $P$ position. \pvs
When $\g=0$, the game tree is as follows.
\begin{center}
    \begin{tikzpicture}
\node (ST1) at (0,0){$\ST{(3\a+1,0,1)}$};
	\node (NS1i) at (0,-1.5){$\NS{(3(\a-1)+2,1,1)}$};
		\node [label={[label distance=1mm]270: B.2.A\tsub{2}}](ST2i) at (-2.5,-3){$\ST{(3(\a-1)+2,0,0)}$};
        \node [label={[label distance=1mm]270: B.2.A\tsub{1}}](ST2ii) at (2.5,-3){$\ST{(3(\a-1)+1,0,2)}$};
            \draw  (ST1) -- node[left=3mm] {M}(NS1i);
            \draw  (NS1i) -- node[left=3mm] {A\tsub{2}}(ST2i);
            \draw  (NS1i) -- node[right=3mm] {A\tsub{1}}(ST2ii);
        \end{tikzpicture}
\end{center}\pvs
When $\a > 1$, B.2.A\tsub{2} is a $P$ position by Theorem \ref{thrm: allina}. When $\a=1$, B.2.A\tsub{1}=$(1,0,2)$ which is a $P$ position by Corollary \ref{coralla}.\pvs
We show in the game tree below that $(3\a+1,0,4\g+1)$ is a $P$ position for $\g \geq 1$.
\begin{center}
    \begin{tikzpicture}
\node (ST1) at (0,0){$\ST{(3\a+1,0,4\g+1)}$};
	\node (NS1i) at (-2.5,-1.5){$\NS{(3(\a-1)+2,1,4\g+1)}$};
		\node [label={[label distance=1mm]270: B.2.A\tsub{2}}](ST2i) at (-5,-3){$\ST{(3(\a-1)+2,0,4\g)}$};
        \node [label={[label distance=1mm]270: B.2.MS\tsub{3}}](ST2ii) at (0,-3){$\ST{(3\a,1,4(\g-1)+3)}$};
    \node (NS1ii) at (5,-1.5){$\NS{(3\a+2,0,4(\g-1)+3)}$};
        \node [label={[label distance=1mm]270: B.2.S\tsub{3}}](ST2iii) at (5,-3){$\ST{(3(\a+1),0,4(\g-1)+1)}$};
            \draw  (ST1) -- node[left=3mm] {M}(NS1i);
            \draw  (ST1) -- node[right=5mm] {S\tsub{3}}(NS1ii);
            \draw  (NS1i) -- node[left=3mm] {A\tsub{2}}(ST2i);
            \draw  (NS1i) -- node[right=3mm] {S\tsub{3}}(ST2ii);
            \draw  (NS1ii) -- node[left=5mm] {M}(ST2ii);
            \draw  (NS1ii) -- node[right] {S\tsub{3}}(ST2iii);
        \end{tikzpicture}
\end{center}\pvs
 By Corollary \ref{cor: nca>c} Part \ref{cor:3k0}, B.2.A\tsub{2} is a $P$ position for all $\a \geq \g +2$ and by Corollary \ref{cor: nca>c} Part \ref{cor:3k1}, B.2.S\tsub{3} is a $P$ position for all $\a \geq \g-2$. By Theorem \ref{thm:wink1=0} Part \ref{cor:3k3}, B.2.MS\tsub{3} is a $P$ position for all $\a \leq \g +1$. Thus, $(3\a+1,0,4\g+1)$ is a $P$ position for all $\a,\g \in \Znn$. \\
\vspace{.25cm}\\
\textit{Proof of Part} \ref{thm:st21}:
Next, we consider the game on $(3\a+2,0,4\g+1)$, with $\a \geq \g$, noting as a base case that $(2,0,1)$ is a $P$ position by Corollary \ref{coralla}. Now, consider the game on an arbitrary $(3\a+2,0,4\g+1)$ with $\a \geq \g$.
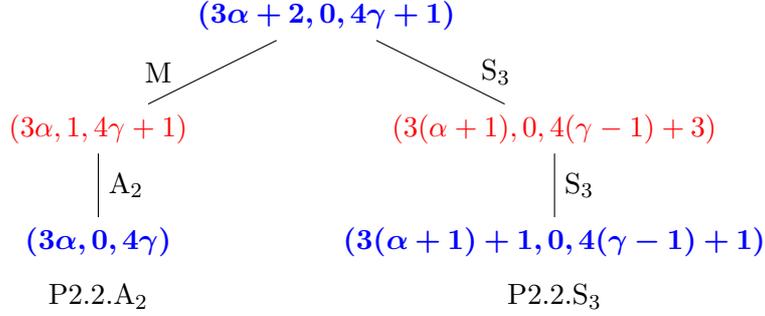
\begin{center}
    \begin{tikzpicture}
\node (ST1) at (0,0){$\ST{(3\a+2,0,4\g+1)}$};
	\node (NS1i) at (-3,-1.5){$\NS{(3\a,1,4\g+1)}$};
		\node [label={[label distance=1mm]270: B.2.A\tsub{2}}](ST2i) at (-3,-3){$\ST{(3\a,0,4\g)}$};
    \node (NS1ii) at (3,-1.5){$\NS{(3(\a+1),0,4(\g-1)+3)}$};
        \node [label={[label distance=1mm]270: B.2.S\tsub{3}}](ST2ii) at (3,-3){$\ST{(3(\a+1)+1,0,4(\g-1)+1)}$};
            \draw  (ST1) -- node[left=4mm] {M}(NS1i);
            \draw  (ST1) -- node[right=4mm] {S\tsub{3}}(NS1ii);
            \draw  (NS1i) -- node[right] {A\tsub{2}}(ST2i);
            \draw  (NS1ii) -- node[right] {S\tsub{3}}(ST2ii);
        \end{tikzpicture}
\end{center}\pvs
Since, $\a \geq \g$, B.2.A\tsub{2} is a $P$ position as proven in Lemma \ref{lemma:B=0} Part \ref{lemma:00} and B.2.S\tsub{3} is a $P$ position by Corollary \ref{cor: nca>c} Part \ref{cor:3k1}.\\
\vspace{.25cm}\\
\vspace{.25cm}\\
\textbf{Proof for} $\bm{c \equiv 2\pmod 4}$.\\
\textit{Proof of Part} \ref{thm:st12}:
Next, we show that $(3\a+1,0,4\g+2)$ is a $P$ position for all $\a \leq \g$, noting as a base case that $(1,0,2)$ is a $P$ position by Corollary \ref{corallc}. We consider the general case below.
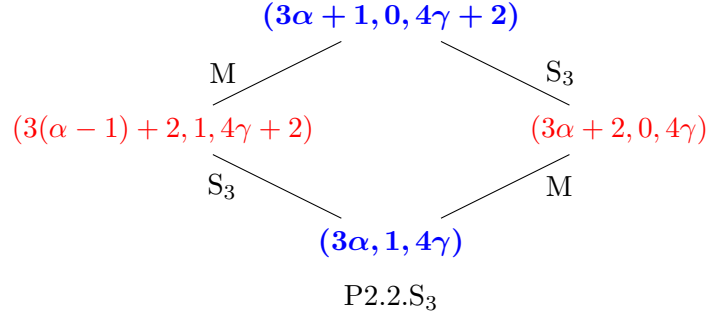
\begin{center}
    \begin{tikzpicture}
\node (ST1) at (0,0){$\ST{(3\a+1,0,4\g+2)}$};
	\node(NS1i) at (-3,-1.5){$\NS{(3(\a-1)+2,1,4\g+2)}$};
        \node [label={[label distance=1mm]270: B.2.S\tsub{3}}](ST2i) at (0,-3){$\ST{(3\a,1,4\g)}$};
    \node (NS1ii) at (3,-1.5){$\NS{(3\a+2,0,4\g)}$};
            \draw  (ST1) -- node[left=4mm] {M}(NS1i);
            \draw  (ST1) -- node[right=4mm] {S\tsub{3}}(NS1ii);
            \draw  (NS1i) -- node[left=4mm] {S\tsub{3}}(ST2i);
            \draw  (NS1ii) -- node[right=4mm] {M}(ST2i);
        \end{tikzpicture}
\end{center}\pvs
We showed that $(3\a,1,4\g)$ is a $P$ position for all $\a \leq \g $ in Theorem \ref{thm:wink1=0} Part \ref{wink2=0}, so $(3\a+1,0,4\g+2)$ is a $P$ position for all $\a \leq \g$.\\
\vspace{.25cm}\\
\textbf{Proof for} $\bm{c \equiv 3\pmod 4}$.\\
\textit{Proof of Part} \ref{thm:st13}:
We show that $(3\a+1,0,4\g+3)$ is a $P$ position for all $\a \leq \g-1$, noting as a base case that $(1,0,7)$ is a $P$ position by Corollary \ref{corallc}. We consider the general case below.
\begin{center}
    \begin{tikzpicture}
\node (ST1) at (0,0){$\ST{(3\a+1,0,4\g+3)}$};
	\node(NS1i) at (3,-1.5){$\NS{(3\a+2,0,4\g+1)}$};
        \node [label={[label distance=1mm]270: B.2.MS\tsub{3}}](ST2i) at (0,-3){$\ST{(3\a,1,4\g+1)}$};
    \node (NS1ii) at (-3,-1.5){$\NS{(3(\a-1)+2,1,4\g+3)}$};
            \draw  (ST1) -- node[right=4mm] {S\tsub{3}}(NS1i);
            \draw  (ST1) -- node[left=4mm] {M}(NS1ii);
            \draw  (NS1i) -- node[right=4mm] {M}(ST2i);
            \draw  (NS1ii) -- node[left=4mm] {S\tsub{3}}(ST2i);
        \end{tikzpicture}
\end{center}\pvs
We showed that $(3\a,1,4\g+1)$ is a $P$ position for all $\a \leq \g -1 $ in Theorem \ref{thm:wink1=0} Part \ref{wink2=1}, so $(3\a+1,0,4\g+3)$ is a $P$ position for all $\a \leq \g-1$.
\vspace{.25cm}\\
\textbf{Proof for} $\bm{c \equiv 0\pmod 4}$.\\
We already proved that a significant portion of Part \ref{thm:k2=0} is true in Lemma \ref{lemma:B=0}. However, these proofs did not consider the possibility of placing $(3\a+1,1,4\a)$ and $(3\a+2,1,4\a)$ so were not fully constructive. Here, we provide a constructive strategy.\\
\vspace{.25cm}\\
\textit{Proof of Part} \ref{thm:st00}:
We first consider the game on $(3\a,0,4\g)$ where $\a \geq \g$. As a base case, see that $(0,0,0)$ is trivially a $P$ position. Then consider the gameplay below.
  \begin{center}
    \begin{tikzpicture}
\node (ST1) at (0,0){$\ST{(3\a,0,4\g)}$};
	\node (NS1i) at (-3,-1.5){$\NS{(3(\a-1)+1,1,4\g)}$};
		\node [label={[label distance=1mm]270: B.2.A\tsub{1}}](ST2i) at (-3,-3){$\ST{(3(\a-1),0,4\g+1)}$};
    \node (NS1ii) at (3,-1.5){$\NS{(3\a+1,0,4(\g-1)+2)}$};
        \node [label={[label distance=1mm]270: B.2.S\tsub{3}}](ST2ii) at (3,-3){$\ST{(3\a+2,0,4(\g-1))}$};
            \draw  (ST1) -- node[left=4mm] {M}(NS1i);
            \draw  (ST1) -- node[right=4mm] {S\tsub{3}}(NS1ii);
            \draw  (NS1i) -- node[right] {A\tsub{1}}(ST2i);
            \draw  (NS1ii) -- node[right] {S\tsub{3}}(ST2ii);
        \end{tikzpicture}
\end{center}\pvs
B.2.A\tsub{1} is a $P$ position by the proof of Part \ref{thm:st01} since $\a \geq \g$. B.2.S\tsub{3} is a $P$ position by Corollary \ref{cor: nca>c} Part \ref{cor:3k0} since $\a \geq \g$. Thus, $(3\a,0,4\g)$ is a $P$ position for all $\a \geq \g$.\\
\vspace{.25cm}\\
\textit{Proof of Part} \ref{thm:st10}:
Next, we consider the game on $(3\a+1,0,4\g)$, noting that as a base case, $(1,0,0)$ is trivially a $P$ position. We show below that $(3\a+1,0,4\g)$ is a $P$ position for all $\a,\g$.
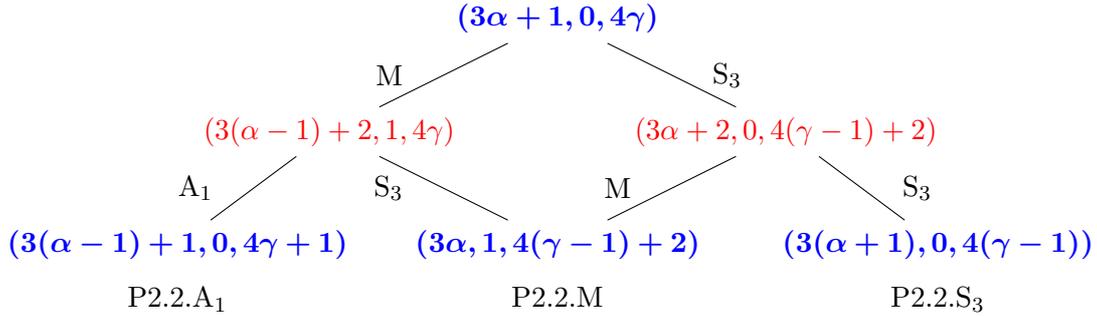
\begin{center}
    \begin{tikzpicture}
\node (ST1) at (0,0){$\ST{(3\a+1,0,4\g)}$};
	\node (NS1i) at (-3,-1.5){$\NS{(3(\a-1)+2,1,4\g)}$};
		\node [label={[label distance=1mm]270: B.2.A\tsub{1}}](ST2i) at (-5,-3){$\ST{(3(\a-1)+1,0,4\g+1)}$};
    \node (NS1ii) at (3,-1.5){$\NS{(3\a+2,0,4(\g-1)+2)}$};
        \node [label={[label distance=1mm]270: B.2.S\tsub{3}}](ST2ii) at (5,-3){$\ST{(3(\a+1),0,4(\g-1))}$};
        \node [label={[label distance=1mm]270: B.2.M}](ST2iii) at (0,-3){$\ST{(3\a,1,4(\g-1)+2)}$};
            \draw  (ST1) -- node[left=4mm] {M}(NS1i);
            \draw  (ST1) -- node[right=4mm] {S\tsub{3}}(NS1ii);
            \draw  (NS1i) -- node[left=4mm] {A\tsub{1}}(ST2i);
            \draw  (NS1ii) -- node[left=4mm] {M}(ST2iii);
            \draw  (NS1i) -- node[left=4mm] {S\tsub{3}}(ST2iii);
            \draw  (NS1ii) -- node[right=4mm] {S\tsub{3}}(ST2ii);
        \end{tikzpicture}
\end{center}\pvs
We showed that B.2.A\tsub{1} is a $P$ position for all $\a,\g$ by the proof of Part \ref{thm:st11}. By the proof of Part \ref{thm:st00}, B.2.S\tsub{3} is a $P$ position for all $\a \geq \g -2$ and by Theorem \ref{thm:wink1=0} Part \ref{wink2=2}, B.2.M is a $P$ position for all $\a \leq \g -3$. Thus, $(3\a+1,0,4\g)$ is a $P$ position for all $\a,\g \in \Znn$. \\
\vspace{.25cm}\\
\textit{Proof of Part} \ref{thm:st20}:
Next, we consider the game on $(3\a+2,0,4\g)$, with $\a \geq \g+1$, noting as a base case that $(5,0,0)$ is a $P$ position by Theorem \ref{thrm: allina}. Now, consider the game on an arbitrary $(3\a+2,0,4\g)$ with $\a \geq \g +1$.
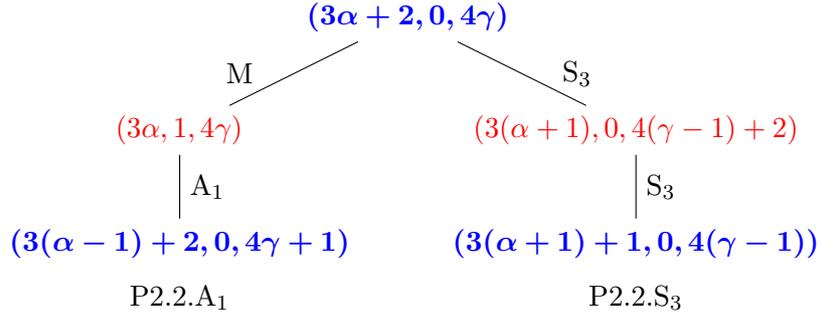
\begin{center}
    \begin{tikzpicture}
\node (ST1) at (0,0){$\ST{(3\a+2,0,4\g)}$};
	\node (NS1i) at (-3,-1.5){$\NS{(3\a,1,4\g)}$};
		\node [label={[label distance=1mm]270: B.2.A\tsub{1}}](ST2i) at (-3,-3){$\ST{(3(\a-1)+2,0,4\g+1)}$};
    \node (NS1ii) at (3,-1.5){$\NS{(3(\a+1),0,4(\g-1)+2)}$};
        \node [label={[label distance=1mm]270: B.2.S\tsub{3}}](ST2ii) at (3,-3){$\ST{(3(\a+1)+1,0,4(\g-1))}$};
            \draw  (ST1) -- node[left=4mm] {M}(NS1i);
            \draw  (ST1) -- node[right=4mm] {S\tsub{3}}(NS1ii);
            \draw  (NS1i) -- node[right] {A\tsub{1}}(ST2i);
            \draw  (NS1ii) -- node[right] {S\tsub{3}}(ST2ii);
        \end{tikzpicture}
\end{center}\pvs
Since $\a \geq \g+1$, B.2.A\tsub{1} is a $P$ position by the proof of Part \ref{thm:st21} and B.2.S\tsub{3} is a $P$ position as shown in the proof of Part \ref{thm:st10}. Then, $(3\a+2,0,4\g)$ is a $P$ position for all $\a \geq \g +1$. \\
\vspace{.25cm}\\
This concludes our proof of Theorem \ref{thm:B=0}.
\end{proof}
\vspace{.25cm}
From here, we give a constructive proof of Lemma \ref{lemma:nck12} and Corollary 5.8 with the information we have built up so far. This ensures that the strategy we outline for $N$ positions in Theorem \ref{thm:B=0N} is in fact constructive. 
\begin{lemma}\label{lemma:ck12}
\textbf{(Constructive.)} Let $(a,b,c)$ be a game state for an $F_4$ Black Hole Zeckendorf game. For all $\alpha,\gamma, k_1,k_3 \in \Znn$, such that $1 \leq k_1\leq 2$, and $0 \leq k_3\leq 3$, $(3\alpha+k_1,1,4\gamma+k_3)$ is an $N$ position. $(3\a,1,4\g)$ is an $N$ position for all $\a \geq \g+1$. 
\end{lemma}
\begin{proof}
    We proceed by demonstrating what the opening move should be from $(3\a+k_1,1,4\g+k_3)$ for $k_1=1,2$. For clarity, we consider each possible combination of $k_1$ and $k_2$ individually. \\
    \vspace{.25cm}\\
\textbf{Proof for} $k_1=1$.\\
    First, consider the game on $(3\a+1,1,4\g)$ in the game tree below. R.1.A\tsub{1} is a $P$ position for all $\a \geq \g-1$ by Theorem \ref{thm:B=0} Part \ref{thm:st01} and R.1.A\tsub{2} is a $P$ position for all $\a \leq \g -2$ by Theorem \ref{thm:B=0} Part \ref{thm:st13}. Therefore, $(3\a+1,1,4\g)$ is an $N$ position for all $\a,\g$. \begin{center}
    \begin{tikzpicture}
\node (ST1) at (0,0){$\ST{(3\a+1,1,4\g)}$};
		\node [label={[label distance=1mm]270: R.1.A\tsub{1}}](NS1i) at (-2.5,-1.5){$\NS{(3\a,0,4\g+1)}$};
        \node [label={[label distance=1mm]270: R.1.A\tsub{2}}](NS1ii) at (2.5,-1.5){$\NS{(3\a+1,0,4(\g-1)+3)}$};
            \draw  (ST1) -- node[left=3mm] {A\tsub{1}}(NS1i);
            \draw  (ST1) -- node[right=3mm] {A\tsub{2}}(NS1ii);
        \end{tikzpicture}
\end{center}\pvs
     Next, consider the game on $(3\a+1,1,4\g+1)$. The next player to move can add in the second and third column to place $(3\a+1,0,4\g)$ which is a $P$ position for all $\a,\g$ by Theorem \ref{thm:B=0} Part \ref{thm:st10}. Therefore, $(3\a+1,1,4\g+1)$ is an $N$ position for all $\a,\g$.\pvs
 Then, consider the game on $(3\a+1,1,4\g+2)$. The next player to move can add in the second and third column to place $(3\a+1,0,4\g+1)$ which is a $P$ position for all $\a,\g$ by Theorem \ref{thm:B=0} Part \ref{thm:st11}. Therefore, $(3\a+1,1,4\g+2)$ is an $N$ position for all $\a,\g$. \pvs
 Lastly, consider the game on $(3\a+1,1,4\g+3)$ in the game tree below. R.1.A\tsub{1} is a $P$ position for all $\a \geq \g+1$ by Theorem \ref{thm:B=0} Part \ref{thm:st00} and R.1.A\tsub{2} is a $P$ position for all $\a \leq \g$ by Theorem \ref{thm:B=0} Part \ref{thm:st12}. Therefore, $(3\a+1,1,4\g)$ is an $N$ position for all $\a,\g$. \begin{center}
    \begin{tikzpicture}
\node (ST1) at (0,0){$\ST{(3\a+1,1,4\g+3)}$};
		\node [label={[label distance=1mm]270: R.1.A\tsub{1}}](NS1i) at (-2.5,-1.5){$\NS{(3\a,0,4(\g+1))}$};
        \node [label={[label distance=1mm]270: R.1.A\tsub{2}}](NS1ii) at (2.5,-1.5){$\NS{(3\a+1,0,4\g+2)}$};
            \draw  (ST1) -- node[left=3mm] {A\tsub{1}}(NS1i);
            \draw  (ST1) -- node[right=3mm] {A\tsub{2}}(NS1ii);
        \end{tikzpicture}
\end{center}\pvs
\textbf{Proof for} $k_1=2$.\\
The strategy here is essentially the same as the one outlined above, with the optimal strategy shifted for the vaule of $k_3$.  First, consider the game on $(3\a+2,1,4\g)$. The next player to move can add in the first and second column to place $(3\a+1,0,4\g+1)$ which is a $P$ position for all $\a,\g$ by Theorem \ref{thm:B=0} Part \ref{thm:st11}. Therefore, $(3\a+2,1,4\g)$ is an $N$ position for all $\a,\g$.\pvs
Next, consider the game on $(3\a+2,1,4\g+1)$ in the game tree below. R.1.A\tsub{1} is a $P$ position for all $\a \leq \g$ by Theorem \ref{thm:B=0} Part \ref{thm:st12} and R.1.A\tsub{2} is a $P$ position for all $\a \geq \g +1$ by Theorem \ref{thm:B=0} Part \ref{thm:st20}. Therefore, $(3\a+2,1,4\g+1)$ is an $N$ position for all $\a,\g$. \begin{center}
    \begin{tikzpicture}
\node (ST1) at (0,0){$\ST{(3\a+2,1,4\g+1)}$};
		\node [label={[label distance=1mm]270: R.1.A\tsub{1}}](NS1i) at (-2.5,-1.5){$\NS{(3\a+1,0,4\g+2)}$};
        \node [label={[label distance=1mm]270: R.1.A\tsub{2}}](NS1ii) at (2.5,-1.5){$\NS{(3\a+2,0,4\g)}$};
            \draw  (ST1) -- node[left=3mm] {A\tsub{1}}(NS1i);
            \draw  (ST1) -- node[right=3mm] {A\tsub{2}}(NS1ii);
        \end{tikzpicture}
\end{center}\pvs
 Then, consider the game on $(3\a+2,1,4\g+2)$ in the game tree below. R.1.A\tsub{1} is a $P$ position for all $\a \leq \g-1$ by Theorem \ref{thm:B=0} Part \ref{thm:st13} and R.1.A\tsub{2} is a $P$ position for all $\a \geq \g$ by Theorem \ref{thm:B=0} Part \ref{thm:st21}. Therefore, $(3\a+2,1,4\g+1)$ is an $N$ position for all $\a,\g$. \begin{center}
    \begin{tikzpicture}
\node (ST1) at (0,0){$\ST{(3\a+2,1,4\g+2)}$};
		\node [label={[label distance=1mm]270: R.1.A\tsub{1}}](NS1i) at (-2.5,-1.5){$\NS{(3\a+1,0,4\g+3)}$};
        \node [label={[label distance=1mm]270: R.1.A\tsub{2}}](NS1ii) at (2.5,-1.5){$\NS{(3\a+2,0,4\g+1)}$};
            \draw  (ST1) -- node[left=3mm] {A\tsub{1}}(NS1i);
            \draw  (ST1) -- node[right=3mm] {A\tsub{2}}(NS1ii);
        \end{tikzpicture}
\end{center}\pvs 
 Then, consider the game on $(3\a+1,1,4\g+3)$. The next player to move can add in the first and second column to place $(3\a+1,0,4(\g+1))$ which is a $P$ position for all $\a,\g$ by Theorem \ref{thm:B=0} Part \ref{thm:st10}. Therefore, $(3\a+2,1,4\g)$ is an $N$ position for all $\a,\g$.
 Finally, see that $(3\a,1,4\g)$ is an $N$ position for all $\a \geq \g+1$ as the next player to move can place $(3(\a-1)+2,0,4\g+1)$ which is a $P$ position for all $\a \geq \g+1$ by Theorem \ref{thm:B=0} Part \ref{thm:st21}. 
\end{proof}
From here, we outline the winnings strategies for the next player to move following an $N$ position. \pvs
\begin{thm}\label{thm:B=0N} 
Let $(a,0,c)=(3\a+k_1,0,4\g+k_3)$ be a board setup for an $F_4$ Black Hole Zeckendorf game. For $\a,\g,k_1,k_3 \in \Znn$, $0 \leq k_1 \leq 2$, and $0 \leq k_3 \leq 3$, the following positions are $N$ positions.
\begin{enumerate}
    \item $c \equiv 1 \pmod 4$: \label{thm:k2=1}
    \begin{enumerate}
        \item $(3\a,0,4\g+1)$ for all $\a \leq \g-2$. \label{thm:ns01}
        \item  $(3\a+2,0,4\g+1)$ for all $\a \leq \g-1$. \label{thm:ns21}
    \end{enumerate}
    \item $c \equiv 2 \pmod 4$: \label{thm:k2=2}
    \begin{enumerate}
        \item $(3\a,0,4\g+2)$ for all $\a,\g$. \label{thm:ns02}
        \item $(3\a+1,0,4\g+2)$ for all $\a \geq \g+1$. \label{thm:ns12}
        \item $(3\a+2,0,4\g+2)$ for all $\a,\g$. \label{thm:ns22}
    \end{enumerate}
    \item $c \equiv 3 \pmod 4$:\label{thm:k2=3}
    \begin{enumerate}
        \item $(3\a,0,4\g+3)$ for all $\a,\g$. \label{thm:ns03}
        \item $(3\a+1,0,4\g+3)$ for all $\a \geq \g$. \label{thm:ns13}
        \item  $(3\a+2,0,4\g+3)$ for all $\a,\g$. \label{thm:ns23}
    \end{enumerate}
    \item $c \equiv 0 \pmod 4$:\label{thm:k2=0}
        \begin{enumerate}
        \item  $(3\a,0,4\g)$ for all $\a \leq \g-1$. \label{thm:ns00}
        \item  $(3\a+2,0,4\g)$ for all $\a \leq \g$. \label{thm:ns20}
    \end{enumerate}
\end{enumerate}
\end{thm}
\begin{proof} Again, we work through each section of the theorem, referencing a base case as an example, and then demonstrating the winning strategy. Here we show that it is always possible for the next player to place a $P$ position. \pvs
\textbf{Proof for} $\bm{c \equiv 1\pmod 4}$.\\
\textit{Proof of Part} \ref{thm:ns01}:
Consider the game on $(3\a,0,4\g+1)$ where $\a \leq \g-2$. As a base case, note that $(0,0,9)$ is an $N$ position by Theorem \ref{thrm: allinc}. Then, consider the tree below.
\begin{center}
    \begin{tikzpicture}
\node (ST1) at (0,0){$\ST{(3\a,0,4\g+1)}$};
	\node (NS1i) at (0,-1.5){$\NS{(3\a+1,0,4(\g-1)+3)}$};
		\node (ST2i) at (3,-3){$\ST{(3\a+2,0,4(\g-1)+1)}$};
            \node [label={[label distance=1mm]270: R.2.M}](NS2i) at (3,-4.5){$\NS{(3\a,1,4(\g-1)+1)}$};
        \node (ST2ii) at (-3,-3){$\ST{(3(\a-1)+2,1,4(\g-1)+3)}$};
            \node [label={[label distance=1mm]270: R.2.A\tsub{1}}](NS2ii) at (-3,-4.5){$\NS{(3(\a-1)+1,0,4\g)}$};
            \draw  (ST1) -- node[right] {S\tsub{3}}(NS1i);
            \draw  (NS1i) -- node[right=4mm] {S\tsub{3}}(ST2i);
            \draw  (NS1i) -- node[left=4mm] {M}(ST2ii);
            \draw  (ST2i) -- node[right] {M}(NS2i);
            \draw  (ST2ii) -- node[right] {A\tsub{1}}(NS2ii);
        \end{tikzpicture}
\end{center}\pvs
R.2.A\tsub{1} is a $P$ position by Lemma \ref{lemma:B=0} Part \ref{lemma:10} and R.2.M is a $P$ position by Theorem \ref{thm:wink1=0} Part \ref{wink2=1}, making $(3\a,0,4\g+1)$ an $N$ position for $\a \leq \g-2$. \\
\vspace{.25cm}\\
\textit{Proof of Part} \ref{thm:ns21}:
Next, we consider the game on $(3\a+2,0,4\g+1)$ with $\a \leq \g-1$. As a base case, note that $(2,0,5)$ is an $N$ position by Corollary \ref{corallc}. Then, note that the first player to move from $(3\a+2,0,4\g+1)$ can merge to place $(3\a,1,4\g+1)$ which is a $P$ position by Theorem \ref{thm:wink1=0} Part \ref{wink2=1} given that $\a \leq \g-1$. Therefore, $(3\a+2,0,4\g+1)$ is an $N$ position for $\a \leq \g-1$.\\
\vspace{.25cm}\\
\textbf{Proof for} $\bm{c \equiv 2\pmod 4}$.\\
\textit{Proof of Part} \ref{thm:ns02}:
We begin by showing that $(3\a,0,4\g+2)$ is an $N$ position for all $\a,\g$, noting as a base case that $(0,0,2)$ is an $N$ position by Theorem \ref{thrm: allinc}. In general, the first player to move from $(3\a,0,4\g+2)$ can split in the third column to place $(3\a+1,0,4\g)$
which is a $P$ position for all $\a,\g$ by Theorem \ref{thm:B=0} Part \ref{thm:st10}. Therefore $(3\a,0,4\g+2)$ is an $N$ position for all $\a,\g$.\\
\vspace{.25cm}\\
\textit{Proof of Part} \ref{thm:ns12}:
The fact that $(3\a+1,0,4\g+2)$ is an $N$ position for $\a \geq \g+1$. First, note that $(6,0,2)$ is an $N$ position by \ref{coralla}. In general, the next player to move from $(3\a+1,0,4\g+2)$ can always split in the $F_3$ column to place $(3\a+2,0,4\g)$ which is a $P$ position for all $\a \geq \g+1$ by Theorem \ref{thm:B=0} Part \ref{thm:st20}.\\
\vspace{.25cm}\\
\textit{Proof of Part} \ref{thm:ns22}:
We conclude this section by showing that $(3\a+2,0,4\g+2)$ is an $N$ position for all $\a, \g$, noting that as a base case $(2,0,2)$ is an $N$ position by Corollary \ref{corallc}. We consider the general case below.
\begin{center}
    \begin{tikzpicture}
\node (ST1) at (0,0){$\ST{(3\a+2,0,4\g+2)}$};
	\node [label={[label distance=1mm]270: R.1.S\tsub{3}}](NS1i) at (3,-1.5){$\NS{(3(\a+1),0,4\g)}$};
 \node [label={[label distance=1mm]270: R.1.M}](NS1ii) at (-3,-1.5){$\NS{(3\a,1,4\g+2)}$};
            \draw  (ST1) -- node[right=4mm] {S\tsub{3}}(NS1i);
            \draw  (ST1) -- node[left=4mm] {M}(NS1ii);
        \end{tikzpicture}
\end{center}\pvs
By Theorem \ref{thm:wink1=0} Part \ref{wink2=2}, R.1.M is a $P$ position for all $\a \leq \g -2$ and by Lemma \ref{lemma:B=0} Part \ref{lemma:00} R.1.S\tsub{3} is a $P$ position for all $\a \geq \g -1$. Thus, $(3\a+2,0,4\g)$ is an $N$ position for all $\a, \g$. \\
\vspace{.25cm}\\
\textbf{Proof for} $\bm{c \equiv 3\pmod 4}$.\\
\textit{Proof of Part} \ref{thm:ns03}:
We begin this section by showing that $(3\a,0,4\g+3)$ is an $N$ position for all $\a,\g$, noting as a base case that $(0,0,3)$ is an $N$ position by Theorem \ref{thrm: allinc}. For the general case, see that from $(3\a,0,4\g+3)$ it is possible to split in the third column to place $(3\a+1,0,4\g+1)$ which is a $P$ position for all $\a,\g$ as shown in Theorem \ref{thm:B=0} Part \ref{thm:st11}. Therefore, $(3\a,0,4\g+3)$ is an $N$ position for all $\a,\g$.\\
\vspace{.25cm}\\
\textit{Proof of Part} \ref{thm:ns13}:
We follow by showing that $(3\a+1,0,4\g+3)$ is an $N$ position for all $\a \geq \g$, noting as a base case that $(1,0,3)$ is an $N$ position by Corollary \ref{corallc}. In the general case, the first player to move from $(3\a+1,0,4\g+3)$ can split in the third column to place $(3\a+2,0,4\g+1)$ which is a $P$ position by Theorem \ref{thm:B=0} Part \ref{thm:st21}. Therefore, $(3\a+1,0,4\g+3)$ is an $N$ position for all $\a \geq \g$.\\
\vspace{.25cm}\\
\textit{Proof of Part} \ref{thm:ns23}:
We conclude this section by showing that $(3\a+2,0,4\g+3)$ is an $N$ position for all $\a, \g$, noting as a base case that $(2,0,3)$ is an $N$ position by Corollary \ref{corallc}. We consider the general case below.
\begin{center}
    \begin{tikzpicture}
\node (ST1) at (0,0){$\ST{(3\a+2,0,4\g+3)}$};
	\node [label={[label distance=1mm]270: R.1.M}](NS1i) at (-3,-1.5){$\NS{(3\a,1,4\g+3)}$};
 \node [label={[label distance=1mm]270: R.1.S\tsub{3}}](NS1ii) at (3,-1.5){$\NS{(3(\a+1),0,4\g+1)}$};
            \draw  (ST1) -- node[left=4mm] {M}(NS1i);
            \draw  (ST1) -- node[right=4mm] {S\tsub{3}}(NS1ii);
        \end{tikzpicture}
\end{center}\pvs
By Theorem \ref{thm:wink1=0} Part \ref{wink2=3}, R.1.M is a $P$ position for all $\a \leq \g +2$ and by Corollary \ref{cor: nca>c} Part \ref{cor:3k1}, R.1.S\tsub{3} is a $P$ position for all $\a \geq \g -1$. Thus, $(3\a+2,0,4\g+3)$ is an $N$ position for all $\a, \g$. \\
\vspace{.25cm}\\
\textbf{Proof for} $\bm{c \equiv 0\pmod 4}$.\\
\textit{Proof of Part} \ref{thm:ns00}:
 Consider the game on $(3\a,0,4\g)$ where $\a \leq \g-1$. As a base case, note that $(0,0,4)$ is an $N$ position by Theorem \ref{thrm: allinc}. In the general case, the first player to move from $(3\a,0,4\g)$ can split in the third column to place $(3\a+1,0,4(\g-1)+2)$ which is a $P$ position by Theorem \ref{thm:B=0} Part \ref{thm:st12} since $\a \leq \g-1$. Therefore, $(3\a,0,4\g)$ is an $N$ position for all $\a \leq \g-1$.
\vspace{.25cm}\\
\textit{Proof of Part} \ref{thm:ns20}:
To conclude the proof of this theorem, we consider the game on $(3\a+2,0,4\g)$ with $\a \leq \g$. As a base case, see that $(2,0,0)$ is an $N$ position by Theorem \ref{thrm: allina}. In the general case, it is possible to merge from $(3\a+2,0,4\g)$ to place $(3\a,1,4\g)$ which is a $P$ position for $\a\leq \g$ by Theorem \ref{thm:wink1=0} Part \ref{wink2=0}. Therefore, $(3\a+2,0,4\g)$ is an $N$ position for all $\a \leq \g$.\\
\vspace{.25cm}\\
This concludes our proof of Theorem \ref{thm:B=0N}.
\end{proof}
\subsection{Empty Board Game}
We now continue to the Empty Board game, to determine which player has a winning strategy for any given $n \in \Zp$. We define the game so that players can still only place pieces in the outermost columns, meaning that no player can move to the $F_2$ column until all pieces have been placed on the board. Then, the initial setup is guaranteed to be in the form of some $(a,0,c)$ above. Players are able to force certain setups, as there are only 2 options for placement, so the ``Tweedledum and Tweedledee" strategy remains an option as in the $F_3$ Black Hole game. Again, players aim to either set the board as a $P$ position or force their opponent to set it as an $N$ position.\par
\begin{thm}\label{F4evensplit}
   Let $(0,0,0)$ be the beginning board for an Empty Board $F_4$ Black Hole Zeckendorf game with $n \in \Zp$ pieces. Players can force certain game setups as outlined below.
   \begin{enumerate}
       \item For any $n \equiv 0 \pmod {4}$, one player can force the game into a setup $(n/4,0,n/4)$. 
       \item For any $n \equiv 1 \pmod {4}$, one player can force the game into a setup $((n-1)/4+1,0,(n-1)/4)$.
       \item For any $n \equiv 2 \pmod {4}$, one player can force the game into a setup $((n-2)/4+2,0,(n-2)/4).$
       \item For any $n \equiv 3 \pmod {4}$, Player 1 can force the game setup $((n-3)/4+3,0,(n-3)/4)$ or can alternatively force the setup $((n-3)/4,0,(n-3)/4+1)$. Player $2$ can force the game to begin with one of these two setups, but cannot force a specific one.
   \end{enumerate}.
\end{thm}
\begin{proof}
    First, consider $n \equiv 0 \pmod {4}$. If $(n/4,0,n/4)$ is an $N$ position for Player $1$ can force Player $2$ to place this setup by placing their first piece in the $F_3$ column, and then placing opposite Player $2$ for all other moves. It will be Player $2$'s turn when there is $1$ piece left, so they will be forced to place down $(n/4,0,n/4)$. If $(n/4,0,n/4)$ is a $P$ position, they can mirror Player $1$ until they are able to place down $(n/4,0,n/4)$.\pvs
    Next, consider $n \equiv 1 \pmod {4}$. If $((n-1)/4+1,0,(n-1)/4)$ is a $P$ position, Player $1$ place their first piece in the $F_1$ column, and then place opposite Player $2$ for all other moves. Every board Player $1$ places will be of the form $(m+1,0,m)$ for some non-negative integer $m$, until all $n$ have been placed, resulting in the setup $((n-1)/4+1,0,(n-1)/4)$. If $((n-1)/4+1,0,(n-1)/4)$ is an $N$ position, Player $2$ can mirror Player $1$ until there is one piece left, forcing Player $1$ to place $((n-1)/4+1,0,(n-1)/4)$. \pvs
    Then, consider $n \equiv 2 \pmod {4}$. If $((n-2)/4+2,0,(n-2)/4)$ is an $N$ position, Player $1$ can place their first piece in the $F_1$ column, and then place opposite Player $2$ for all other moves. Every board Player $1$ places will be of the form $(m+1,0,m)$ until there is only $1$ piece left, which Player $2$ is forced to set down in the $F_1$ column, placing down $((n-2)/4+2,0,(n-2)/4)$. If this is a $P$ position, Player $2$ can force it via the ``Tweedledum and Tweedledee" strategy until there are only $2$ pieces left. Since pieces can only be placed in the outermost columns, Player $1$ is forced to place $((n-2)/4+1,0,(n-2)/4)$, allowing Player $2$ to place $((n-2)/4+2,0,(n-2)/4)$. \pvs
    Lastly, consider $n \equiv 3 \pmod {4}$. If Player 1 places their first piece in the $F_1$ column, and then places opposite Player $2$ for all other moves, then every board Player $1$ places will be of the form $(m+1,0,m)$ until there are only $2$ pieces left; Player $2$ is forced to set down one in the $F_1$ column, and Player $1$ sets the other, placing down $((n-3)/4+3,0,(n-3)/4)$. Player $1$ can also force the setup $((n-3)/4,0,(n-3)/4+1)$ by placing their first piece in the $F_3$ column, and then mirroring Player $2$, until the board results in Player $1$ placing $((n-3)/4,0,(n-3)/4+1)$. Player $2$ can force the game to be one of these positions by mirroring Player $1$ until there are $3$ pieces left. However, at this point Player $1$ determines whether the setup is $((n-3)/4+3,0,(n-3)/4)$ or $((n-3)/4,0,(n-3)/4+1)$ as noted above.\pvs
\end{proof}
\begin{thm}\label{thm:win04}
    Player $2$ has a constructive strategy for winning an Empty Board $F_4$ Black Hole Zeckendorf game for any $n \equiv 0,2,4,6,9,11,13 \pmod{16}$ such that $n\neq 2, 32$ in which case Player $1$ has the winning strategy. Player $1$ has a constructive strategy for winning an Empty Board $F_4$ Black Hole Zeckendorf game for any $n \equiv 1,3,5,7,8,10,12,14,15 \pmod{16}$, such that $n\neq 17,47$, in which case Player $2$ has the winning strategy.
\end{thm}
\begin{proof}
We split this proof into sections based on the value of $n$ modulo 4.\\
\vspace{.25cm}\\
\textbf{Proof for }\bm{$n \equiv 0 \pmod 4$}.\par
    Player $2$ sets down the board as $(n/4,0,n/4)$, by Theorem \ref{F4evensplit} so they win when this is a $P$ position and lose when it is an $N$ position. For $n\equiv 0,4 \pmod{16},$ we have $n/4 \equiv 0,1 \pmod 4$ while $n\equiv 8,12 \pmod{16}$ implies $n/4 \equiv 2,3 \pmod 4$. \pvs
 For all $n/4\geq 9$,  with $0 \leq k_1 \leq 2$, $0 \leq k_3 \leq 3$ there is no $\a\leq \g$ such that $3\a+k_1=n/4=4\g+k_3$. Thus, when $n/4 \equiv 0,1 \pmod 4$, Player $2$ wins for all $n/4 \geq 9$ by Theorem \ref{thm:B=0} Part \ref{thm:k2=0} and \ref{thm:k2=1} respectively. When $n/4 \equiv 2,3 \pmod 4$, Player $1$ wins for all $n/4 \geq 9$ by Theorem \ref{thm:B=0N} Part \ref{thm:k2=2} and \ref{thm:k2=3} respectively. \pvs
    Then, it is only necessary to explicitly consider the cases when $n/4\leq 8$. We outline the outline the value of $n$, the cases, the winners, and the parts of Theorem \ref{thm:B=0} and Theorem \ref{thm:B=0N} that determines the winner below. \begin{align*}
      &\text{Value of }n &&\text{Board Setup} &&& \text{Winner} &&&& \text{Theorem}\\
       &n=4 &&(1,0,1)=(3(0)+1,0,4(0)+1) &&& \text{Player 2} &&&& \ref{thm:B=0} \text{ Part } \ref{thm:st11}\\
       &n=8 &&(2,0,2)=(3(0)+2,0,4(0)+2) &&& \text{Player 1} &&&& \ref{thm:B=0N} \text{ Part } \ref{thm:ns22}\\
       &n=12 &&(3,0,3)=(3(1)+0,0,4(0)+3) &&& \text{Player 1} &&&& \ref{thm:B=0N} \text{ Part } \ref{thm:ns03}\\
       &n=16 &&(4,0,4)=(3(1)+1,0,4(1)+0) &&& \text{Player 2} &&&& \ref{thm:B=0} \text{ Part } \ref{thm:st10}\\
       &n=20 &&(5,0,5) =(3(1)+2,0,4(1)+1)&&& \text{Player 2} &&&& \ref{thm:B=0} \text{ Part } \ref{thm:st21}\\
       &n=24 &&(6,0,6)=(3(2)+0,0,4(1)+2) &&& \text{Player 1} &&&& \ref{thm:B=0N} \text{ Part } \ref{thm:ns02}\\
       &n=28 &&(7,0,7)=(3(2)+1,0,4(1)+3) &&& \text{Player 1} &&&& \ref{thm:B=0N} \text{ Part } \ref{thm:ns13}\\
       &n=32 &&(8,0,8)=(3(2)+2,0,4(2)+0) &&& \text{Player 1} &&&& \ref{thm:B=0N} \text{ Part } \ref{thm:ns20}
    \end{align*}\pvs
    Thus, Player $2$ wins for all $n\equiv 0,4 \pmod{16}$ such that $n\neq 32$ and Player $1$ wins for all $n \equiv 8,12 \pmod{16}$, as well as $n=32$.\\
    \vspace{.25cm}\\
    \textbf{Proof for }\bm{$n \equiv 1 \pmod 4$}.\par
    Player $1$ sets down the board as $((n-1)/4+1,0,(n-1)/4)$ by Theorem \ref{F4evensplit}. For $n\equiv 1,5 \pmod{16},$ we have $(n-1)/4 \equiv 0,1 \pmod 4$ while $n\equiv 9,13 \pmod{16}$ implies $(n-1)/4 \equiv 2,3 \pmod 4$. \pvs
    For all $(n-1)/4\geq 5$ such that $0 \leq k_1 \leq 2$, $0 \leq k_3 \leq 3$ there is no $\a\leq \g$ such that $3\a+k_1=(n-1)/4+1$ and $4\g+k_3= (n-1)/4$. Therefore, when $(n-1)/4 \geq5$, Player $1$ wins for all $(n-1)/4 \equiv 0,1 \pmod 4$, by Theorem \ref{thm:B=0} Parts \ref{thm:k2=1} and \ref{thm:k2=0}. When $(n-1)/4 \equiv 2,3 \pmod 4$, Player $2$ wins for all $(n-1)/4$ by Theorem \ref{thm:B=0N} Parts \ref{thm:k2=2} and \ref{thm:k2=3}.\pvs
    Then, it is only necessary to explicitly consider the cases when $(n-1)/4\leq 4$. We outline the value of $n$, the cases, the winners, and the part of Theorems \ref{thm:B=0} and \ref{thm:B=0N} that determines the winner below. \begin{align*}
    &\text{Value of }n &&\text{Board Setup} &&& \text{Winner} &&&& \text{Theorem}\\
       &n=1 &&(1,0,0) = (3(0)+1,0,4(0)+0) &&& \text{Player 1} &&&& \ref{thm:B=0} \text{ Part } \ref{thm:st10}\\
       &n=5 &&(2,0,1) = (3(0)+2,0,4(0)+1)&&& \text{Player 1} &&&& \ref{thm:B=0} \text{ Part } \ref{thm:st21}\\
       &n=9 &&(3,0,2) = (3(1)+0,0,4(0)+2)&&& \text{Player 2} &&&& \ref{thm:B=0N} \text{ Part } \ref{thm:ns02}\\
       &n=13 &&(4,0,3)= (3(1)+1,0,4(0)+3) &&& \text{Player 2} &&&& \ref{thm:B=0N} \text{ Part } \ref{thm:ns13}\\
       &n=17 &&(5,0,4) = (3(1)+2,0,4(1)+0)&&& \text{Player 2} &&&& \ref{thm:B=0N} \text{ Part } \ref{thm:ns20}
    \end{align*}\pvs
    Therefore, Player $1$ wins for all $n\equiv 1,5 \pmod{16}$ such that $n\neq 17$ and Player $2$ wins for all $n \equiv 9,13 \pmod{16}$, as well as $n=17$.
 \vspace{.25cm}\\
    \textbf{Proof for }\bm{$n \equiv 2 \pmod 4$}.\par
    Player $2$ sets down the board as $((n-2)/4+2, 0, (n-2)/4)$ by Theorem \ref{F4evensplit}. For $n\equiv 2,6 \pmod{16},$ we have$ (n-2)/4 \equiv 0,1 \pmod 4$ while $n\equiv 10,14 \pmod{16} $ implies$ (n-2)/4 \equiv 2,3 \pmod 4$. For all $(n-2)/4\geq 1$, and  $0 \leq k_1 \leq 2$, $0 \leq k_3 \leq 3$ there is no $\a\leq \g$ such that $3\a+k_1=(n-2)/4+2$ and $4\g+k_3=(n-2)/4$. When $(n-2)/4 \equiv 0,1 \pmod 4$, Player $2$ wins for all $(n-2)/4 \geq 1$ by Theorem \ref{thm:B=0} Parts \ref{thm:k2=1} and \ref{thm:k2=0}.When $(n-2)/4 \equiv 2,3 \pmod 4$, Player $1$ wins for all $(n-1)/4 \geq 1$ by Theorem \ref{thm:B=0N} Parts \ref{thm:k2=2} and \ref{thm:k2=3}.\pvs
    Then, it is only necessary to explicitly consider the case $(2,0,0)$, which Player $1$ wins by Theorem \ref{thrm: allina}. Hence, Player $2$ wins for all $n\equiv 2,6 \pmod{16}$ such that $n\neq 2$ and Player $1$ wins for all $n \equiv 10,14 \pmod{16}$, as well as $n=2$.\\
 \vspace{.25cm}\\
    \textbf{Proof for }\bm{$n \equiv 3 \pmod 4$}.\par
    By Theorem \ref{F4evensplit}, Player $1$ either sets down $((n-3)/4+3, 0,(n-3)/4)$ or $((n-3)/4,0,(n-3)/4+1)$.  \pvs
    For $n\equiv 3,7 \pmod{16},$ we have $(n-3)/4 \equiv 0,1 \pmod 4$. Here, Player $1$ should set down $((n-3)/4+3, 0,(n-3)/4)$. Then, for all $(n-3)/4\geq 0$, and $0 \leq k_1 \leq 2$, $0 \leq k_3 \leq 3$ there is no $\a\leq \g$ such that $3\a+k_1=(n-3)/4+3$ and $4\g+k_3=(n-3)/4$. So, for all $n\equiv 3,7 \pmod{16}$, Player $1$ wins by Theorem \ref{thm:B=0} Parts \ref{thm:k2=1} and \ref{thm:k2=0}.\pvs
    Next, for $n \equiv 15 \pmod {16},$ we have $(n-3)/4+1 \equiv 0 \pmod 4$. Player $1$ should set down $((n-3)/4,0,(n-3)/4+1)$; then, for all $(n-3)/4 \geq 12$, and $0 \leq k_1 \leq 2$, $0 \leq k_3 \leq 3$ there is no $\a\leq \g$ such that $3\a+k_1=(n-3)/4$ and $4\g+k_3=(n-3)/4+1$. Then, Player $1$ wins by Theorem \ref{thm:B=0} Part \ref{thm:k2=0}. We explicitly consider $n\equiv 15 \pmod {16}$ such that $(n-3)/4 \leq 11$ below.
    \begin{align*}
    &\text{Value of }n &&\text{Board Setup} &&& \text{Winner} &&&& \text{Theorem}\\
       &n=15 &&(3,0,4) = (3(1)+0,0,4(1)+0) &&& \text{Player 1} &&&& \ref{thm:B=0} \text{ Part } \ref{thm:st00}\\
       &n=31 &&(7,0,8)= (3(2)+1,0,4(2)+0) &&& \text{Player 1} &&&& \ref{thm:B=0} \text{ Part } \ref{thm:st10}\\
       &n=47 &&(11,0,12) = (3(3)+2,0,4(3)+0)&&& \text{Player 2} &&&& \ref{thm:B=0N} \text{ Part } \ref{thm:ns20}\\
        &n=47 &&(14,0,11) = (3(4)+2,0,4(2)+3)&&& \text{Player 2} &&&& \ref{thm:B=0N} \text{ Part } \ref{thm:ns23}
    \end{align*}\pvs
    Lastly, for $n\equiv 11 \pmod{16},$ we have $(n-3)/4 \equiv 2 \pmod 4$ and $(n-3)/4+1 \equiv 3 \pmod 4$.\\
    Therefore, Player $2$ can force Player $1$ to set the board such that $c \equiv 2,3 \pmod 4$. If Player $1$ places $((n-3)/4+3, 0,(n-3)/4)$, then since there exists no $\a\leq \g$ that satisfies $3\a+k_1=(n-3)/4+3$ and $4\g+k_3=(n-3)/4$, Player $2$ wins by Theorem \ref{thm:B=0} Part \ref{wink2=1}.\pvs
    If Player $1$ places $((n-3)/4,0,(n-3)/4+1)$, Player $2$ wins for all $(n-3)/4 \geq 12$ by Theorem \ref{thm:B=0} Part \ref{thm:k2=3}. We explicitly consider $n\equiv 11 \pmod {16}$ such that $(n-3)/4 \leq 11$ below.
    \begin{align*}
    &\text{Value of }n &&\text{Board Setup} &&& \text{Winner} &&&& \text{Theorem}\\
       &n=11 &&(2,0,3)=(3(0)+2,0,4(0)+3) &&& \text{Player 2} &&&& \ref{thm:B=0N} \text{ Part } \ref{thm:ns23}\\
       &n=27 &&(6,0,7)=(3(2)+0,0,4(1)+3) &&& \text{Player 2} &&&& \ref{thm:B=0N} \text{ Part } \ref{thm:ns03}\\
       &n=43 &&(10,0,11)=(3(3)+1,0,4(2)+3) &&& \text{Player 2} &&&& \ref{thm:B=0N} \text{ Part } \ref{thm:ns13}
    \end{align*}\pvs
    So then, we find that Player $1$ has a constructive winning solution for all $n \equiv 1,3,5,7,8,10,12,14,15 \pmod {16}$ such that $n\neq 17,47$, for which Player $2$ has a winning solution. Player $2$ has a constructive winning solution for all $n \equiv 0,2,4,6,9,11,13 \pmod {16}$ such that $n\neq 2,32$, for which Player $1$ has a winning solution.
    \end{proof}
\section{Future Work}
To determine a constructive solution for the Empty Board $F_4$ Black Hole Zeckendorf Game as we defined it, it was only necessary to determine which player wins for any given $(a,0,c)$. Still, a complete analysis of $(a,b,c)$ for any value of $b$ could give more insight into the original problem, and could be an interesting area for future work. Redefining the Empty Board  $F_4$ Black Hole Zeckendorf Game to allow players to place in the $F_2$ column would also be interesting, as the TweedleDee-TweedleDum strategy is no longer as applicable. \pvs
Note also that the strategy of reducing modulo $F_m$ is not as immediately successful when the black hole is at $F_m$ such that $m\geq 5$. See below an example of a game with a black hole on $F_5=8$.
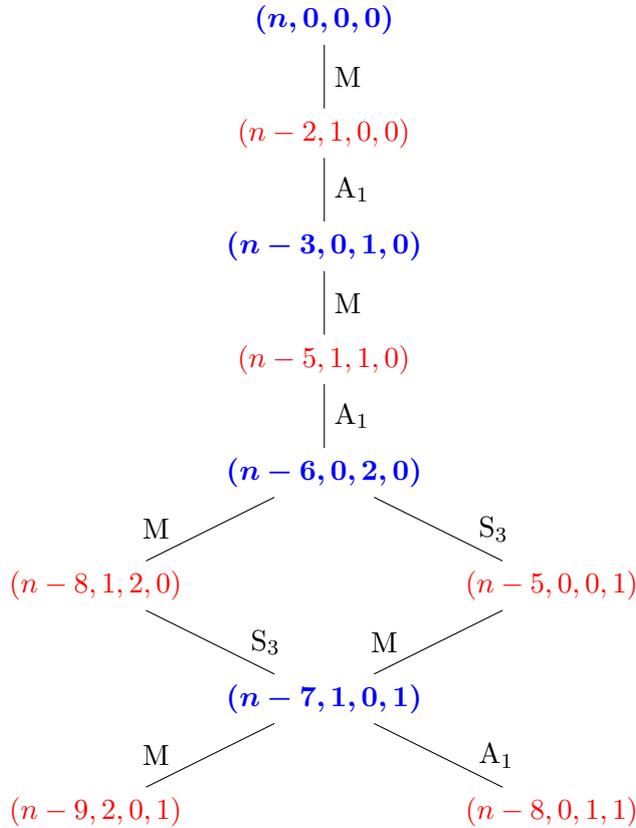
\begin{center}
    \begin{center}
    \begin{tikzpicture}
\node (ST1) at (0,0){$\ST{(n,0,0,0)}$};
	\node(NS1) at (0,-1.5){$\NS{(n-2,1,0,0)}$};
        \node (ST2) at (0,-3){$\ST{(n-3,0,1,0)}$};
            \node(NS2) at (0,-4.5){$\NS{(n-5,1,1,0)}$};
                \node (ST3) at (0,-6){$\ST{(n-6,0,2,0)}$};
                    \node(NS3i) at (3,-7.5){$\NS{(n-5,0,0,1)}$};
                    \node(NS3ii) at (-3,-7.5){$\NS{(n-8,1,2,0)}$};
                        \node (ST4) at (0,-9){$\ST{(n-7,1,0,1)}$};
                             \node(NS4i) at (3,-10.5){$\NS{(n-8,0,1,1)}$};
                              \node(NS4ii) at (-3,-10.5){$\NS{(n-9,2,0,1)}$};

            \draw  (ST1) -- node[right] {M}(NS1);
            \draw  (NS1) -- node[right] {A\tsub{1}}(ST2);
            \draw  (ST2) -- node[right] {M}(NS2);
            \draw  (NS2) -- node[right] {A\tsub{1}}(ST3);
            \draw  (ST3) -- node[right=4mm] {S\tsub{3}}(NS3i);
            \draw  (ST3) -- node[left=4mm] {M}(NS3ii);
            \draw  (NS3i) -- node[left=4mm] {M}(ST4);
            \draw  (NS3ii) -- node[right=4mm] {S\tsub{3}}(ST4);
            \draw  (ST4) -- node[right=4mm] {A\tsub{1}}(NS4i);
            \draw  (ST4) -- node[left=4mm] {M}(NS4ii);
        \end{tikzpicture}
\end{center}
    \end{center}\pvs
In the typical Zeckendorf game with $n=8$, Player $1$ is forced to place $(n-8,0,1,1)$. When we play on $4$ columns but with more than $8$ pieces,  Player $1$ can now also place $(n-9,2,0,1)$, making it challenging for Player $2$ to place a piece in the black hole without giving Player $1$ the opportunity to do so first. Expanding to black holes on higher Fibonacci numbers may very well prove a fruitful step in determining a constructive solution to the original Zeckendorf game.
 \newpage

\section*{Acknowledgement(s)}

We are extremely grateful for the detailed comments from the referees, which led to significantly improved expositions in many places and fixed errors in some of the proofs. Special thanks to Paul Baird-Smith, whose code we edited to play through the Zeckendorf Black Hole game and to Kyle Burke for his guidance on combinatorial game theory notation.

\section*{Funding}

Authors were supported by Williams College, The Finnerty
Fund,  The College of William \& Mary Charles Center, The Cissy Patterson Fund, and NSF Grant DMS2241623.

\bigskip

\appendix
\section{Base Case Games for \texorpdfstring{$(a,0,0)$}{(a,0,0)}}\label{bcaseallina}

For the base cases of Theorem \ref{thrm: allina}, we have the following games, showing that $(a,0,0)$ is a $P$ position for $a=1,3,4,5,7$ and is an $N$ position for $a=2$.
\begin{center}
    \begin{tikzpicture}

    \node(ST1) at (0,0){$\ST{(1,0,0)}$};
    \node(ST2) at(2,0){$\ST{(2,0,0)}$};
    \node(NS2) at (2,-1.5){$\NS{(0,1,0)}$};
    \node(ST3) at (4,0){$\ST{(3,0,0)}$};
        \node(NS3) at (4,-1.5){$\NS{(1,1,0)}$};
            \node(ST3i) at (4,-3){$\ST{(0,0,1)}$};
    \node(ST4) at (6,0){$\ST{(4,0,0)}$};
        \node(NS4) at (6,-1.5){$\NS{(2,1,0)}$};
            \node(ST4i) at (6,-3){$\ST{(1,0,1)}$};

    \node(ST5) at (8,0){$\ST{(5,0,0)}$};
    \node(NS5) at (8,-1.5){$\NS{(3,1,0)}$};
        \node(ST5i) at (8,-3){$\ST{(2,0,1)}$};
            \node(NS5i) at (8,-4.5){$\NS{(0,1,1)}$};
                \node(ST5ii) at (8, -6){$\ST{(0,0,0)}$};
    \node(ST7) at (10, 0){$\ST{(7,0,0)}$};
        \node(NS7) at (10,-1.5){$\NS{(5,1,0)}$};
            \node(ST7i) at (10,-3){$\ST{(4,0,1)}$};
                \node(NS7i) at (10,-4.5){$\NS{(2,1,1)}$};
                    \node(ST7ii) at (10, -6){$\ST{(1,0,2)}$};
                        \node(NS7ii) at (10, -7.5){$\NS{(2,0,0)}$};
                            \node(ST7iii) at (10, -9){$\ST{(0,1,0)}$};

    \draw (ST2) -- node[right]{M}(NS2);

    \draw (ST3) -- node[right]{M}(NS3);
    \draw (NS3) -- node[right]{A\tsub{1}}(ST3i);

    \draw (ST4) -- node[right]{M}(NS4);
    \draw (NS4) -- node[right]{A\tsub{1}}(ST4i);

    \draw (ST5) -- node[right]{M}(NS5);
    \draw (NS5) -- node[right]{A\tsub{1}}(ST5i);
    \draw (ST5i) -- node[right]{M}(NS5i);
    \draw (NS5i) -- node[right]{A\tsub{2}}(ST5ii);

    \draw (ST7) -- node[right]{M}(NS7);
    \draw (NS7) -- node[right]{A\tsub{1}}(ST7i);
    \draw (ST7i) -- node[right]{M}(NS7i);
    \draw (NS7i) -- node[right]{A\tsub{1}}(ST7ii);
    \draw (ST7ii) -- node[right]{S\tsub{3}}(NS7ii);
    \draw (NS7ii) -- node[right]{M}(ST7iii);
    \end{tikzpicture}
    \end{center}

\section{Base Case Games for \texorpdfstring{$(0,0,c)$}{(0,0,c)}}\label{bcaseallinc}
For the base cases of Theorem \ref{thrm: allinc}, we have the following games, showing that $(1,0,c)$ is a $P$ position for $c=2,4,8$ and an $N$ position for $c=3$. Note that $(1,0,0)$ and $(1,0,1)$ are trivially $P$ positions. 
\begin{center}
    \begin{tikzpicture}
        \node(NS4) at (0,0){$\ST{(1,0,2)}$};
            \node(ST4i) at (0,-1.5){$\NS{(2,0,0)}$};
                \node(NS4i) at (0,-3){$\ST{(0,1,0)}$};

        \node(NS5) at (3,0){$\ST{(1,0,3)}$};
            \node(ST5i) at (3,-1.5){$\NS{(2,0,1)}$};
                \node(NS5i) at (3,-3){$\ST{(0,1,1)}$};
                    \node(ST5ii) at (3,-4.5){$\NS{(0,0,0)}$};

    \draw (NS4) -- node[right]{S\tsub{3}}(ST4i);
    \draw (ST4i) -- node[right]{M}(NS4i);

    \draw (NS5) -- node[right]{S\tsub{3}}(ST5i);
    \draw (ST5i) -- node[right]{M}(NS5i);
    \draw (NS5i) -- node[right]{A\tsub{2}}(ST5ii);

        \node(NS6) at (6,0){$\ST{(1,0,4)}$};
            \node(ST6i) at (6,-1.5){$\NS{(2,0,2)}$};
                \node(NS6i) at (6,-3){$\ST{(3,0,0)}$};
                    \node(ST6ii) at (6,-4.5){$\NS{(1,1,0)}$};
                        \node(NS6ii) at (6,-6){$\ST{(0,0,1)}$};

        \node(NS10) at (9,0) {$\ST{(1,0,8)}$};
            \node(ST10i) at (9, -1.5){$\NS{(2,0,6)}$};
                \node(NS10i) at (9,-3){$\ST{(3,0,4)}$};
                \node(ST10iia) at (8,-4.5){$\NS{(1,1,4)}$};
                    \node(NS10iia) at (8, -6){$\ST{(0,0,5)}$};
                \node(ST10iib) at (10, -4.5){$\NS{(4,0,2)}$};
                    \node(NS10iib) at (10, -6){$\ST{(5,0,0)}$};

    \draw (NS6) -- node[right]{S\tsub{3}}(ST6i);
    \draw (ST6i) -- node[right]{S\tsub{3}}(NS6i);
    \draw (NS6i) -- node[right]{M}(ST6ii);
    \draw (ST6ii) -- node[right]{A\tsub{1}}(NS6ii);

    \draw (NS10) -- node[right]{S\tsub{3}}(ST10i);
    \draw (ST10i) -- node[right]{S\tsub{3}}(NS10i);
    \draw (NS10i) -- node[left=3mm]{M}(ST10iia);
    \draw (ST10iia) -- node[right]{A\tsub{1}}(NS10iia);
    \draw (NS10i) -- node[right=3mm]{S\tsub{3}}(ST10iib);
    \draw (ST10iib) -- node[right]{S\tsub{3}}(NS10iib);

    \end{tikzpicture}
\end{center}
\section{Base Case Games for Lemma \ref{lemma:nck12}:}\label{basecasea1c}
We showed in Lemma \ref{lemma:nck12} that $(3\a+1,1,4\g+k_3)$ is an $N$ position for all $4\g+\a+k_3-1 \neq 3$ and that $(3\a+2,1,4\g+k_3)$ is an $N$ position for all $4\g+\a+k_3+1 \neq 3$. We show here that the exception cases are also $N$ positions. Note that $\g \neq 0$ can only be true when in the first equation $\g=1, \a=k_3=0$. Here, the board is set as $(1,1,4)$, from which the next player can place $(0,0,5)$, winning by Theorem \ref{thrm: allinc}. For all other cases, $\g=0$ since $\a,\g,k_3 \in \Znn$. \pvs
First, see that $(3\a+1,1,4\g+k_3)$ is an $N$ position for $\a+k_3 = 4$. The possible solutions $(\a, k_3)$ such that $\a,k_3 \in \Znn$ and $0 \leq k_3 \leq 3$ are $(1,3)$, $(2,2)$, $(3,1)$, and $(4,0)$. We show that the corresponding game positions $(4,1,3)$, $(7,1,2)$, $(10,1,1)$, and $(12,0,1)$ are $N$ positions in the game trees below. 
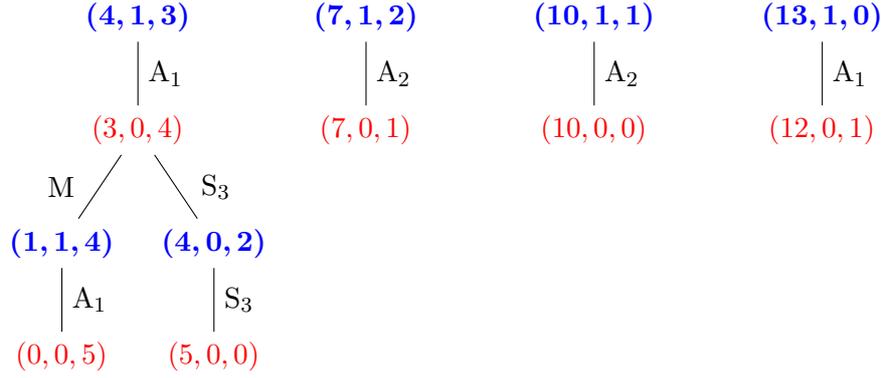
\begin{center}
    \begin{tikzpicture}
\node (ST1) at (-4,0){$\ST{(4,1,3)}$};
	\node (NS1) at (-4,-1.5){$\NS{(3,0,4)}$};
        \node (ST1i) at (-3,-3){$\ST{(4,0,2)}$};
            \node (NS1i) at (-3,-4.5){$\NS{(5,0,0)}$};
        \node (ST1ii) at (-5,-3){$\ST{(1,1,4)}$};
            \node (NS1ii) at (-5,-4.5){$\NS{(0,0,5)}$};

 \node (ST2) at (-1,0){$\ST{(7,1,2)}$};
	\node (NS2) at (-1,-1.5){$\NS{(7,0,1)}$};

 \node (ST3) at (2,0){$\ST{(10,1,1)}$};
	\node (NS3) at (2,-1.5){$\NS{(10,0,0)}$};

 \node (ST4) at (5,0){$\ST{(13,1,0)}$};
	\node (NS4) at (5,-1.5){$\NS{(12,0,1)}$};
            \draw  (ST1) -- node[right] {A\tsub{1}}(NS1);
            \draw  (NS1) -- node[right=2mm] {S\tsub{3}}(ST1i);
            \draw  (NS1) -- node[left=2mm] {M}(ST1ii);
             \draw  (ST1i) -- node[right] {S\tsub{3}}(NS1i);
             \draw  (ST1ii) -- node[right] {A\tsub{1}}(NS1ii);
            \draw  (ST2) -- node[right] {A\tsub{2}}(NS2);
            \draw  (ST3) -- node[right] {A\tsub{2}}(NS3);
            \draw  (ST4) -- node[right] {A\tsub{1}}(NS4);
        \end{tikzpicture}
\end{center}
$(5,0,0)$ and $(10,0,0)$ win by Theorem \ref{thrm: allina}. $(0,0,5)$ wins by Theorem \ref{thrm: allinc}. $(7,0,1)$ and $(12,0,1)$ win by Corollary $\ref{coralla}$. Therefore, $(3\a+2,1,4\g+k_3)$ is an $N$ position.\pvs
Next consider $(3\a+2,1,4\g+k_3)$ which must be an $N$ position for all $\a+k_3 = 2$ when $\a,k_3 \in \Znn$ and $0 \leq k_3 \leq 3$. The possible solutions $(\a, k_3)$ to $\a+k_3= 2$ such that $\a,k_3 \in \Znn$ are $(0,2)$, $(1,1)$ and $(2,0)$. We show that the corresponding game positions $(2,1,2)$, $(5,1,1)$, and $(8,1,0)$ are all $N$ positions. 
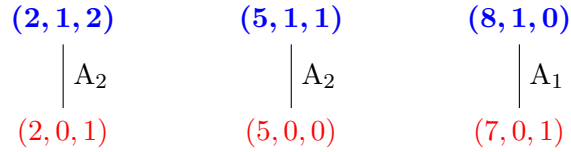
\begin{center}
    \begin{tikzpicture}
\node (ST1) at (-3,0){$\ST{(2,1,2)}$};
	\node (NS1) at (-3,-1.5){$\NS{(2,0,1)}$};

 \node (ST2) at (0,0){$\ST{(5,1,1)}$};
	\node (NS2) at (0,-1.5){$\NS{(5,0,0)}$};

 \node (ST3) at (3,0){$\ST{(8,1,0)}$};
	\node (NS3) at (3,-1.5){$\NS{(7,0,1)}$};
            \draw  (ST1) -- node[right] {A\tsub{2}}(NS1);
            \draw  (ST2) -- node[right] {A\tsub{2}}(NS2);
            \draw  (ST3) -- node[right] {A\tsub{1}}(NS3);
        \end{tikzpicture}
\end{center}
$(2,0,1)$ and $(7,0,1)$ win by Corollary $\ref{coralla}$ and $(5,0,0)$ wins by Theorem \ref{thrm: allina}. Therefore, $(3\a+2,1,4\g+k_3)$ is an $N$ position.\pvs

\section{Code}
In this project, we modified code by Paul Baird-Smith for \cite{BEFM1} to function for the Black Hole Zeckendorf game. Our code can be found here:  \href{https://github.com/JennaShuff/ZeckendorfGameBlackHoles2024}{https://github.com/JennaShuff/ZeckendorfGameBlackHoles2024} while Baird-Smith's code is available at \href{https://github.com/paulbsmith1996/ZeckendorfGame/}{https://github.com/paulbsmith1996/ZeckendorfGame/} .

\end{document}